\documentclass[11pt]{article}
\usepackage{fullpage}
\usepackage{amsmath, amsthm, amssymb}
\usepackage[textsize=small]{todonotes}
\usepackage[hyperfootnotes=false]{hyperref}
\usepackage[capitalize]{cleveref}
\usepackage{natbib}

\usepackage{ulem}
\usepackage{xr}

\input{macro}

\crefrangeformat{assumption}{\textup{#3(#1)#4--#5(#2)#6}}
\crefmultiformat{assumption}{Assumptions~#2#1#3}{ and~#2#1#3}{, #2#1#3}{ and~#2#1#3}
\crefrangeformat{assumption}{Assumptions~#3#1#4--#5#2#6}
\Crefrangeformat{assumption}{Assumptions~#3#1#4--#5#2#6}

\theoremstyle{plain}

\newtheorem{theorem}{Theorem}[section]
\newtheorem{lemma}[theorem]{Lemma}
\newtheorem{proposition}[theorem]{Proposition}
\newtheorem{corollary}[theorem]{Corollary}

\theoremstyle{definition}

\newtheorem{assumption}{Assumption}
\newtheorem*{remark}{Remark}

\newcommand{\rev}[1]{#1}
\newcommand{\diag}{\mathrm{diag}}

\title{Hadamard--Langevin dynamics for sampling the $\ell_1$-prior}

\author{
  Ivan Cheltsov\thanks{University of Bath. \texttt{ic488@bath.ac.uk}} \and
  Federico Cornalba\thanks{University of Bath. \texttt{fc402@bath.ac.uk}} \and
  Clarice Poon\thanks{University of Warwick. \texttt{clarice.poon@warwick.ac.uk}} \and
  Tony Shardlow\thanks{University of Bath. \texttt{t.shardlow@bath.ac.uk}}
}

\date{}

\begin{document}

\maketitle

\begin{abstract}
Priors with non-smooth log-densities, such as the $\ell_1$-prior, are widely used in Bayesian
inverse problems for their sparsity-inducing properties. Existing Langevin-based sampling methods
typically rely on proximal mappings or smooth approximations, which alter the target distribution.
We propose an alternative approach based on a Hadamard product parameterization of the
$\ell_1$-norm, leading to a smooth but nonconvex and non--globally Lipschitz potential whose
marginal law exactly recovers the desired posterior. The resulting Hadamard--Langevin dynamics
(HLD) defines a diffusion process that is analytically distinct from proximal or mirror-type
Langevin schemes. Our main contribution is a rigorous well-posedness theory for both the
continuous and discrete HLD. We establish existence and uniqueness of strong solutions, geometric
ergodicity of the continuous dynamics, and convergence of the discretized scheme as the step size
tends to zero. These results provide the first theoretical foundation for sampling from nonconvex,
nonsmooth posteriors through overparameterized Langevin dynamics.
\end{abstract}



\section{Introduction}

Sampling from sparsity-promoting distributions is a central task in high-dimensional Bayesian inference and inverse problems. A prototypical example is
\begin{align}  
\rho(x)=
\frac 1 Z\exp\bigl(-\beta\bigl(\lambda \|x\|_1 + G(x)\bigr)\bigr),\qquad x\in \mathbb{R}^d,\label{target}
\end{align}
where  $Z$ is a normalization constant, $\beta>0$ is the inverse temperature, $\lambda$ is a regularization parameter, and $G\colon\real^d\to\real$ represents the data likelihood. Due to the non-smoothness of the $\ell_1$-term, existing Langevin-type methods typically rely on smoothing techniques such as the Moreau--Yosida envelope or proximal mappings.

In this work, we propose and theoretically analyse the Hadamard--Langevin dynamics (HLD) --- a novel smooth overparameterized dynamics whose marginal stationary distribution exactly recovers $\rho$ without smoothing and bias. 
The key idea is to exploit the Hadamard representation of the $\ell_1$-norm,
\begin{align}\label{hadamard}
\norm{x}_1 = \sum_i \abs{x_i} = \min_{u\odot v = x} \left\{\frac12 \norm{u}^2_2 + \frac12 \norm{v}^2_2\right\}.
\end{align}
\rev{By replacing the $\ell_1$ term, the log density in \cref{target} is recast as a smooth but non-concave function in $(u,v)$ due to the Hadamard product $u\odot v=x$.} This representation, which has recently garnered interest in optimization to study implicit regularization and to develop efficient optimization solvers, has not yet been explored from a sampling perspective. We demonstrate that applying Langevin dynamics in the lifted space $(u,v)$ leads to a well-defined diffusion process whose invariant measure projects exactly onto $\rho$.

\subsection{The Hadamard--Langevin dynamics}

Our starting point is the classical Gaussian-mixture identity underlying the hierarchical Bayesian formulation of the Laplace prior (\cite{west1987scale}),
\begin{equation}\label{eq:laplace}
   \exp\pa{-a\abs{z}} = \frac{a}{ \sqrt{2\pi}}\int_0^\infty \frac{1}{\sqrt{\eta}} \exp\pa{ -\frac{z^2}{2 \eta} - \frac{a^2 \eta}{2}} d\eta, \qquad a\in \real_+,\,\, z\in\real.
\end{equation}
Introducing latent variables $\eta_i>0$ for each coordinate of $x$, one can rewrite the target distribution~\eqref{target} as the marginal of a joint law over $(x,\eta)\in\mathbb{R}^d\times\mathbb{R}_+^d$:
\begin{equation}\label{eq:gibbs}
    \tilde \pi(x,\eta) \propto \frac{1}{\prod_{i=1}^d \sqrt{\eta_i}} \exp\Biggl( - \sum_{i=1}^d\left(\frac{x_i^2}{2\eta_i} + \frac{\beta^2 \lambda^2 \eta_i }{2 }\right) - \beta \,G(x)\Biggr).
\end{equation}

Since  sampling from $\tilde\pi$ yields the correct marginal $\rho(x)$,
\cite{park2008bayesian} proposed a Gibbs sampler where one alternates between sampling $x$ and $\eta$. However, such works typically restrict to the case where $G$ is quadratic, so sampling $x$ corresponds to sampling a Gaussian, and this in turn requires inverting a potentially high-dimensional matrix.
In contrast, as we explain next, we focus on a reparameterized formulation of the above, which will allow us to derive a Langevin algorithm, under general assumptions on the data term $G$.

\subsubsection{From hierarchical models to Hadamard overparameterization}\label{sec:intro_hadamard}
A key observation is that the Gaussian mixture representation can be viewed as the stochastic analogue of the quadratic variational form for the $\ell_1$-norm.
By applying  a change of variables of  $$x = u \odot v,\quad \eta = u \odot u /(\lambda \beta)$$ to \eqref{eq:gibbs}, we obtain  the Hadamard parametrized joint law 
\begin{align}
\pi(u,v)=\frac{1}{Z_\pi} \prod_{i=1}^d  u_i\, \exp\pp{-\beta\, \pp{\frac12\lambda \|(u,v)\|^2+G(u\odot v)} }\label{pi}
\end{align}
over $(u,v)\in \mathcal X\eqdef \real_+^d\times \mathbb {R}^d$ for normalization constant $Z_\pi$ and $\norm{(u,v)}^2=\norm{u}^2+\norm{v}^2$. The Hadamard  product $x=u\odot v$ follows the same distribution $\rho$ defined in \cref{target}. This is formalized in the following statement.

\begin{theorem}\label{thm:overparam}If $(u,v)\sim \pi$ as defined in~\cref{pi}, then $x=u \odot v\sim\rho$ as defined in~\cref{target}. In particular, for any smooth $\phi:\mathbb{R}^d \xrightarrow[]{} \mathbb{R}$,
    \[\int_{\mathcal X} \phi(u \odot v)\, \pi(u,v)\,du\,dv = \int_{\mathbb{R}^
    d} \phi(x) \,\rho(x)\,dx.\]
    Furthermore, the normalization factors $Z_\pi$ and $Z$ of $\pi$ and $\rho$, respectively, satisfy
 $\frac{1}{Z} = \frac{1}{Z_\pi} \left(\sqrt{\frac{\pi}{2\beta \lambda}}\right)^d.$
\end{theorem}

Note that the log-density is smooth except for $\log(u_i)$ terms, which we show can be handled easily.   It is therefore natural to derive Langevin dynamics (see \eqref{main}) to sample $\pi$, and recover  samples of $\rho$ by evaluating $u\odot v$.

\subsection{Main contributions}\label{sec:contri}
By treating $\rho$ using the Hadamard parameterization to obtain a differentiable log-density, we demonstrate that we can apply Langevin sampling methods to good effect. We provide a complete theoretical analysis of the Hadamard--Langevin dynamics by establishing well-posedness and stability of the continuous and discretized dynamics. 
\begin{enumerate}
    \item We establish the well-posedness on $\mathcal{X} = \real_+^d\times \mathbb {R}^d$ of the Langevin system  associated with $\pi$ given in \eqref{pi}. This includes geometric convergence to the invariant measure (\cref{sec:cont}). We present two approaches: The first approach removes the singularity in the $\log(u)$-term via a Cartesian change of coordinates, to transform our drift into a locally bounded term.  The second approach  exploits connections of our Langevin system to the well-studied Cox--Ingersoll--Ross and Ornstein--Uhlenbeck  processes  via a Girsanov transformation. The Girsanov transformation also ensures that $u_t$ will be positive almost surely.
    \item We develop a numerical scheme to approximate the Langevin system and show the convergence of the numerical scheme to $\pi$ (\cref{sec:disc}). We establish strong convergence of the method as well as discuss convergence of the stationary distribution of the numerical method. One of the challenges in the analysis is that the negative log-density is neither convex nor globally Lipschitz; despite this, we are able to prove geometric ergodicity of the numerical scheme by relying on two tools, namely a coupling technique for different realisations of the numerical method inside  sufficiently large balls, and a rapidly-decaying bound for the escape probability from such balls.

\end{enumerate}
As mentioned,  the Gaussian mixture formula \cref{eq:laplace} can be seen as the equivalent of the quadratic variational form for sampling and 
our work can be seen as a first attempt to deploy this trick for Langevin dynamics. Even though we restrict to sampling the $\ell_1$-prior, as discussed in the conclusion, we believe that our work offers a new perspective on sampling sparse priors and future work would aim to generalize this to other structure-inducing priors.

\subsection{Related background}

\subsubsection*{Quadratic variational forms}

It is well-known that many of the sparsity-inducing regularizers, and in particular, the $\ell_1$-norm, have \textit{quadratic variational forms} \citep{black1996unification,bach2012optimization}. In particular, if the regulariser $R(x) = \phi(x\odot x)$ for $x\in\real^d$ and  a concave function $\phi$, one can rewrite  using quadratic variational forms 
\begin{equation}\label{eq:eta}
R(x) = \inf_{\eta\in\real^d_+} \sum_{i=1}^d \frac{x_i^2}{\eta_i} + \psi(\eta),
\end{equation}
for some function $\psi$. 
This includes the $\ell_q$ (semi)norms for $q\in (0,2]$ and various nonconvex regularizers such as the Lorentzian function \citep{black1996unification}; a similar formula also holds for functionals on matrices, such as the nuclear norm \citep{bach2012optimization}. 
In the case of $R(x) = \norm{x}_1$, one can choose $\psi(\eta) = \eta/4$. 

This idea has been exploited to good effect in optimization, and the computer-vision community found the new variable $\eta$  can help increase the curvature of a `flat' landscape and has led to graduated non-convexity algorithms \citep{black1996unification}.
This lifting of the $\ell_1$-norm to the $x$ and $\eta$ variables is the basis of  iteratively reweighted least-squares algorithms \citep{daubechies2008iteratively}, where one alternates between solving least squares problems on $x$ and exact solves on $\eta$. 

\subsubsection*{The Hadamard product parameterization}
More recently, there has been renewed interest in the Hadamard form of these quadratic variational formulas \eqref{eq:eta}, particularly in the context of convex structure-inducing regularizers.
For $\psi(\eta) = \eta/4$, by making a change of variable with $x = u\odot v$ and $\eta/2 = u\odot u$, one can equivalently write \eqref{eq:eta} as $ \norm{x}_1 = \inf_{x = u\odot v} \{\frac12\norm{u}^2 + \frac12\norm{v}^2$\} where $\norm{\cdot}$ is the Euclidean 2-norm.
While the formulation in \eqref{eq:eta} retained convexity in both variables $x$ and $\eta$,
this Hadamard product-like formula is no longer convex but is now smooth; see \cref{fignew}. Consequently, one can make use of   powerful smooth optimization solvers such as quasi-Newton methods. One of the first works that replace the $\ell_1$-norm with the Hadamard product is the work of \cite{hoff2017lasso}; see also \cite{Poon2023-av,poon2021smooth} for further discussions on exploiting these forms for optimization. In the case of the $\ell_1$-norm, one can show that this nonconvexity is in fact benign \citep{poon2021smooth}. Furthermore, it can be shown that gradient descent with the Hadamard parameterization is implicitly linked to mirror descent with a hyperbolic entropy potential \citep{Poon2023-av}. This points to a potential benefit of such parameterizations for modern optimization when working on high-dimensional and ill-conditioned problems, where \cite{chizat2022convergence} investigated the superior convergence behaviour of such mirror descent algorithms in the context of sparse measures.
Finally, over-parameterization and its implications for implicit bias in gradient-descent training of neural networks have also been of wide interest in the machine learning community. In particular, the Hadamard product parameterization, even without explicit regularization, causes gradient descent to bias towards sparse solutions \citep{vaskevicius2019implicit,chou2023more}.  

\subsubsection*{Unadjusted Langevin algorithms (ULA) and Moreau--Yosida envelopes}\label{mye}

 Langevin dynamics is a simple way to sample from the posterior distribution $\rho(x)$ \citep{Roberts1996-je}. 
Unadjusted Langevin Algorithm (ULA), in the case where $\rho(x) \propto\exp(-\beta H(x))$ for a \textit{smooth} potential $H\colon \real^d\to \real$, is the update rule 
$
x^{n+1} = x^n -  \nabla H(x^n)\tstep + \sqrt{2\tstep/\beta} \,\xi^{n},
$
for independent samples $\xi^{n} \sim\Nn(0,\Id_d)$ and  time step $\tstep>0$. This can be interpreted as a discretization $x^n\approx x_{n\tstep}$ of the overdamped Langevin stochastic differential equation (SDE):
$
dx_t = - \nabla H(x_t)\,dt + \sqrt{2/\beta}\,dW_t,
$
where 
$W_t$ is a standard $\real^d$-valued Wiener process. 

In the case where $H(x) = G(x) + R(x)$ and $R$ is non-smooth  (as in our case, see \eqref{target}), a canonical approach is to replace $R$ by a smooth approximation defined by the Moreau--Yosida envelope $R_\gamma$ \citep{pereyra2016proximal}:
$
R_\gamma(x)\eqdef \min_{z\in \RR^d} (R(z) + \frac{1}{2\gamma}\norm{x-z}^2).
$
The function $R_\gamma$ is then $1/\gamma$-smooth and converges pointwise to $R$ as $\gamma \to 0$. 
One can then directly apply ULA to $H_\gamma\eqdef G + R_\gamma$, leading to the Moreau--Yosida regularized Unadjusted Langevin Algorithm (MYULA). Convergence to the stationary distribution requires one to take $\gamma \to 0$ \citep{durmus2018efficient}. Under suitable conditions, such as convexity of $G$ and Lipschitz smoothness, convergence results like geometric ergodicity of the stationary distribution   $\rho_\gamma \propto \exp(-\beta H_\gamma)$ have been established, as well as the convergence of $\rho_\gamma$ to $\rho$ as $\gamma \to 0$. 
While early works replace $R$ by its Moreau envelope, \cite{durmus2019analysis} removes this approximation bias by directly incorporating the proximal operator of $R$.
 Another line of work that avoids the bias from the Moreau-envelope smoothing directly samples from $H$ via subgradient Langevin \citep{fruehwirth2024ergodicity} in the case where $H$ is strongly convex. In general, smoothing by the Moreau envelope allows one to leverage all existing theory  for smooth Langevin dynamics, while more recent works  derive cleaner algorithms and proofs by heavily relying on techniques from convex analysis.

\begin{figure} 
\includegraphics[width=\linewidth]{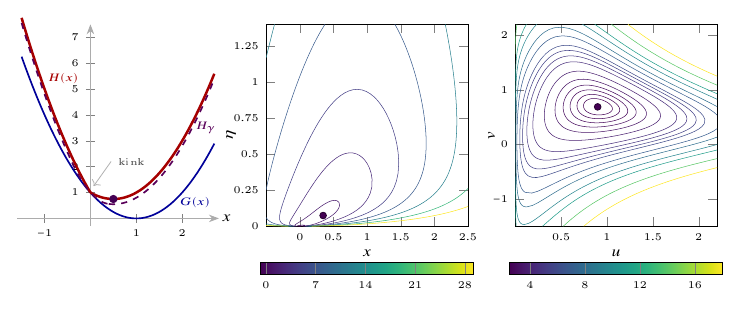}

\caption{Left: potential functions $G(x)$ (blue), $H(x)$ with non-smooth regulariser $R(x)=\lambda|x|$ (red with kink at $x=0$),  and $H_\gamma(x)$ with the smooth Moreau regulariser  (dashed).  Centre: contours of the potential $\frac1{2\beta} \log\eta +x^2/2\beta\eta+\beta \lambda^2 \eta/2+ G(x)$ corresponding to \cref{eq:gibbs}. Lifting to $(x,\eta)$-space via  the Gaussian mixture representation yields a smooth joint density, but with highly anisotropic contours. Right: contours of the potential $-(1/\beta)\log u+\frac12\lambda(u^2+v^2)+G(uv)$ after the Hadamard change of variable. 
The Hadamard parameterization produces a smooth, well-conditioned potential whose Langevin dynamics are  tractable.}\label{fignew}
\end{figure}

\section{Continuous dynamics}\label{sec:cont}
We work with  $G$ under the following assumptions.

\begin{assumption}[data term $G$]\label{reg1} Assume that (i) $G$ is bounded below, that (ii) $G\colon \real^d\to \real$ is continuously differentiable, that  (iii) $x^\top\nabla G(x)$ is bounded below (so $x^\top\nabla G(x) \geq -\clower$ for some $\clower$), and that (iv)  $\|\nabla G(x)\|_\infty$ is bounded uniformly for $x\in \real^d$ (so $\cinfty := \sup_x\|\nabla G(x)\|_\infty < \infty$).
\end{assumption}

The data term $G$ is bounded below so that $\rho$ in \eqref{target} is a probability measure for some normalization constant $Z>0$. Note that (iii) holds if $G$ is a convex function. Important examples that satisfy (i), (ii) and (iii) include the quadratic loss $G(x) = \norm{ A x - y}^2$ for an $m\times d$ matrix $A$ and $y\in \real^m$; the logistic loss $G(x)=\sum_{i=1}^m \log( 1+ \exp(-y_i A_i^\top x) )$ for $A_i\in\real^d$; and smooth approximations of the unit loss,  such as $G(x) = \sum_i \sigma(y_i A_i^\top x)$ for $\sigma(x)=\arctan(x)$ or $\sigma(x)=1/(1+\exp(x))$.
Note that  $G(x) = \sum_i \sigma(y_i A_i^\top x)$  satisfies condition (iv) for any function $\sigma$ with bounded derivatives, which includes the logistic loss and smooth versions of the unit loss.  The least-squares loss $G(x)=\norm{Ax-y}^2$ does not satisfy condition (iv) on the boundedness of $\nabla G$, and one might consider instead the Huber loss. However, in our numerical computations, the method behaves well with the quadratic loss, so assumption (iv) is likely a limitation of our theoretical analysis rather than the method itself.

\subsection{Langevin sampling}
We will use Langevin sampling to produce samples from $\pi$. Write $\pi\propto\exp(-\beta\, H)$ for $H(u,v)=\frac 12 \,\lambda\, \norm{(u,v)}^2 + G(u\odot v)- \frac{1}{\beta}\,\sum_i\log u_i$. The gradient is
$
\nabla H(u,v)=\lambda\begin{pmatrix}
    u\\v
\end{pmatrix}
+\begin{pmatrix}
    v\\ u
\end{pmatrix}\odot \nabla G(u\odot v)
-\begin{pmatrix}
    1/\beta u\\ 0
\end{pmatrix},
$
where we interpret $1/u\in\real^d$ as the vector with entries $1/u_i$. Therefore, the corresponding Langevin equations are 
\begin{equation}
\begin{split}
du_t&=-\lambda\,   u_t\,dt - v_t\odot \nabla G(u_t\odot v_t)\,dt
 +\frac{1}{\beta \,u_t}\,dt+ \sqrt{2/\beta}\,dW^1_t,\\
dv_t&=-\lambda\,  v_t\,dt -  \,u_t \odot \nabla G(u_t\odot v_t)\,dt+  \sqrt{2/\beta}\,dW^2_t,
\label{main}
\end{split}
\end{equation}
for independent $\real^d$-valued Brownian motions $W^1_t$ and $W^2_t$.  
It is well known that 
  $\pi$ given in \eqref{pi} is an invariant distribution of \cref{main} and we explore ergodicity in \cref{erg}. We will show that \cref{main} is well posed on $\mathcal{X}=\real_+^d\times \mathbb R^d$.

The variables $(u,v)$ in \cref{main} can be viewed as the cylindrical co-ordinates of an axially symmetric distribution $\rho(t,(u,v))$ with axial co-ordinate $v\in\real^d$ and radial co-ordinate $u\in[0,\infty)^d$. We reformulate the same equations in Cartesian coordinates $(y,z,v)\in \real^{3d}$, where $u=(y\odot y+z\odot z)^{1/2}$ (for the element-wise square root). For $\vec x=[y,z,v]\in \real^{3d}$, consider
\begin{equation}  
 \begin{aligned}
 d\vec x= \Bigl[-\lambda \vec x - \nabla G_{3}(\vec x)\Bigr]\,dt +\sqrt{2/\beta}\, dW^3(t),
  \end{aligned}\label{cyl}
\end{equation}
for
$G_{3}(\vec x)= G( (y\odot y + z\odot z)^{1/2} \odot v)$ and an $\mathbb{R}^{3d}$-valued Brownian motion $W^3(t)$. 
The function $G_{3}$ is locally bounded under \cref{reg1}, and solutions of \cref{cyl} are well defined. We use these as a tool to analyze \cref{main}.

Let $P_t(X_0,B)=\mathbb{P}(\vec x_t\in B\mid \vec x_0=X_0)$.
Recall that a process is irreducible if for any Borel set $B$ with positive Lebesgue measure and any $X_0$, there exists $t>0$ so that $P_t(X_0,B)>0$, and strong Feller if $P_t(X_0,B)
$ is continuous in $X_0$ for all Borel sets $B$.
We say a function $\mu(y,z,v)$ is cylindrical on $\real^{3d}$ if for some $\tilde \mu\colon \real^{2d}\to \real$ we have $\mu(y,z,v)=\tilde \mu(u,v)$  for all $u=((y\odot y)+(z\odot z))^{1/2}$.

\begin{theorem} \label{wp}  Suppose that \cref{reg1} holds. Subject to initial condition $(u,v)_0=(u_0,v_0)\in \mathcal{X}=\real_+^d\times \mathbb R^d$,
there is a unique solution $(u, v)_t$ to \cref{main} such that  $u_t>0$ almost surely and $\mean{\sup_{0\le t\le T} \norm{(u,v)_t}^q}<\infty$, for any $T>0$, $q\ge 2$. Further, $(u,v)_t$ is strong Feller and irreducible on $\mathcal X$.

The probability density function $p(t,\cdot)$ of $\vec x_t$ (see \cref{cyl}) is well-defined. 
If the initial distribution is cylindrical, the marginal distribution of $((y_t^2+z_t^2)^{1/2},v_t)$ equals the law of the solution $(u_t,v_t)$ of \cref{main}.
\end{theorem}
\begin{proof}
Under \cref{reg1}, $\nabla G_3$ is bounded on $\real^{3d}$: because $\nabla G$ is
bounded by \cref{reg1}(iv), and the chain-rule factors $y_i/u_i$ and $z_i/u_i$
(where $u_i=(y_i^2+z_i^2)^{1/2}$) are uniformly bounded by $1$. Hence, the drift
$-\lambda\vec{x}-\nabla G_3(\vec{x})$ of \cref{cyl} is bounded and measurable on
$\real^{3d}$. Further, the diffusion coefficient $a=(2/\beta)I_{3d}$ is constant and
uniformly elliptic, so conditions~(1.3) and~(1.4) of
\cite{Stroock2013-df} hold trivially.
The coercivity bound
$$
\vec{x}^\top\bigl(-\lambda\vec{x}-\nabla G_3(\vec{x})\bigr)
\leq -\lambda\|\vec{x}\|^2+K,
$$
which follows from \cref{reg1}(iii), together with It\^{o}'s formula applied to
$\|\vec{x}_t\|^2$ and Gronwall's inequality, gives
$\mathbb{E}[\sup_{s\le t}\|\vec{x}^{(n)}_s\|^2] \le C$ uniformly in the
localisation level $n$, for each $t>0$. Chebyshev's inequality then yields the
non-explosion condition~(1.7) of \cite[Theorem~10.1.3]{Stroock2013-df}.
Consequently, \cite[Corollary~10.1.4]{Stroock2013-df} gives that the martingale
problem for $(a, b)$ is well-posed, that $\vec x_t$ is strong Feller, and that
$\vec x_t$ admits a probability density function absolutely continuous with
respect to Lebesgue measure on $\real^{3d}$ for each $t>0$.
Higher $q$-moments $\mathbb{E}[\sup_{s\le t}\|\vec x_s\|^q] < \infty$ follow
analogously by applying It\^{o}'s formula to $(\|\vec x_t\|^2)^{q/2}$.
Irreducibility of $\vec{x}_t$ follows from the non-degenerate additive noise by
standard arguments~\cite[Section~6]{Bellet2006-pq}.

 The generator for \cref{main} has singular coefficients on a half space and regularity of the density does not follow directly from standard theories; we show that the $u$-equation in \cref{main} arises from the solution $\vec x_t$ of \cref{cyl} and that the process $(u,v)_t$ inherits properties from $\vec x_t$. Since the noise in \cref{cyl} is additive and non-degenerate, the process
$\mathbf{x}_t$ has a probability density function absolutely continuous with respect to
Lebesgue measure on $\mathbb{R}^{3d}$ for each $t>0$
\cite[Corollary~10.1.4]{Stroock2013-df}, so in particular $P(u_{i,t}=0) =
P((y_{i,t},z_{i,t})=(0,0)) = 0$ for all $t>0$.
We may therefore apply the It\^{o} formula to $u=(y\odot y+z\odot z)^{1/2}$,
following \cite[Proposition~3.21]{Karatzas1991-xr} (approximating the square root by
smooth functions and passing to the limit, justified since $u_{i,t}>0$ a.s.):
$$
\begin{aligned}
  du_t&=\frac{1}{u_t}\odot(y_t\odot(-\lambda\,y_t)+z_t\odot(-\lambda\, z_t))\,dt
  -\frac{1}{u_t}(y_t\odot (y_t/u_t)+z_t\odot(z_t/u_t))\odot v_t\odot \nabla G(u_t\odot v_t)\,dt \\
  &\quad\quad+\sqrt\frac{2}{\beta} \frac{1}{u_t} \odot\begin{pmatrix}y_t\\z_t\end{pmatrix}
  \cdot\begin{pmatrix} dW_t^{3,y}\\ dW_t^{3,z}\end{pmatrix}
  + \frac{1}{2} \frac{2}{\beta} \left(\frac{-1}{u^3}\odot(y_t\odot y_t+z_t\odot z_t)
  +\frac{2}{u_t}\right)\,dt\\
  &=-\lambda\, u_t\,dt - v_t\odot \nabla G(u_t\odot v_t)\,dt
  +\frac{1}{\beta\,u_t}\,dt+ \sqrt\frac{2}{\beta}\,dW_t,
\end{aligned}
$$
where the $\mathbb{R}^d$-valued process $W(t):= u_t^{-1} \odot(y_t,z_t)
\cdot (dW_t^{3,y}, dW_t^{3,z})$ is again a Brownian motion due to the identity
$u^2_t = y^2_t + z^2_t$ \cite[Theorem~3.3.16]{Karatzas1991-xr}.
This is the $u$-equation in \cref{main} as required. By the uniqueness of the solutions, we see that \cref{cyl} and \cref{main} agree  as long as the initial distributions agree.

The process $(u,v)_t$ inherits several properties from $\vec x_t$:  namely, $(u,v)_t$ is well posed in $\mathcal X$ and is strong Feller and irreducible, and the distribution $\rho(t,\cdot)$  of $(u,v)_t$ is given by $p(t,\cdot)$ if the initial distribution of $\vec x$ is cylindrical. Further, the moment bounds on $\vec x_t$  imply the required moment bounds for $(u,v)_t$.
 \end{proof}

 Let $L^q(\pi)$ denote the set of integrable functions $\phi\colon\mathcal X\to\real$ such that $\int \abs{\phi(x)}^q \pi(x)\,dx<\infty,$ where $\pi$ is the invariant measure \cref{pi}. In the following,  we show how to control expectations of $\phi((u,v)_t)$, which extends the moment bounds in \cref{wp} to $\phi\in L^q(\pi)$, and includes some $\phi$ that are singular near $u_i=0$.

\begin{proposition}\label{l1}Let $(u,v)_t$ denote the solution of \cref{main} with initial distribution $\rho_0$. Let $p\geq 1$. Assume that ($p,q)$ are conjugate and that  $h_0=\rho_0/\pi$ belongs to $L^p(\pi)$. If $\phi \in L^q(\pi)$, then $\mean{ \phi((u,v)_t)}\le \norm{\phi}_{L^q(\pi)}\, \norm{h_0}_{L^p(\pi)}$ for $t>0$.
\end{proposition}
\begin{proof} Let $\rho(t,\cdot)$ denote the density of $(u,v)_t$. Denote $h=\rho/\pi$ and $h_0=\rho_0/\pi$. It is well known (e.g., \cite{Markowich2004-dw}) that $h$ satisfies the PDE
\[
\frac{\partial h}{\partial t}=\Delta h + \nabla \log \pi\cdot \nabla h,\qquad h(0,z)=h_0(z),
\]
on $t>0$ and $z\in(0,\infty)^d\times \real^d$. For any twice-differentiable convex function $\psi$, it follows~\cite[Eq.~9]{Markowich2004-dw}   that
\[
\frac{d}{dt} \int \psi(h(t,x))\, \pi(x)\,dx=-\int \psi''(h(t,x))\,|\nabla h(t,x)|^2\,\pi(x)\,dx\le 0.\]
Choosing $\psi(h)=h^p$ implies that $\|h(t,\cdot)\|_{L^p(\pi)}$ is a non-increasing function of time.
Thus, writing
 \begin{align*}
     \mean{\phi((u,v)_t)}
 =\int \phi(x)\,h(t,x)\, \pi(x)\, dx
 \le \norm{\phi}_{L^{q}(\pi)} \, \norm{h(t,\cdot)}_{L^p(\pi)} \le \norm{\phi}_{L^{q}(\pi)} \, \norm{h_0}_{L^p(\pi)}
 \end{align*}
 completes the proof.
\end{proof}

As a corollary, we prove the following moment bounds on $1/u_{1,t}^r$ for $r<2$. The density of the invariant measure $\pi$ behaves like $u_1$ near $u_1=0$, and we expect $\mathbb{E}\bp{1/u_{1,t}^2}$ to be infinite. 
\begin{corollary}\label{cor} Let the assumption of \cref{wp} hold and the initial distribution $\rho_0(u,v)$ be such that $\rho_0/\pi$ is uniformly bounded.
Then, for $i=1,\dots,d$,

1. $\displaystyle\sup_{t>0}\mathbb{E} \bp{\frac 1{u_{i,t}^r}}<\infty$ for $0<r<2$ (here $u_{i,t}$ denotes the $i$th component of $u_t$).\vspace{1ex}

2. For each $T>0$ and $\epsilon>0$, we have
 $\mean{\pp{\int_0^T \frac{1}{u_{i,s}^{2-\epsilon}}\,ds}^q} <\infty $ for all $1\le q<2$.
  
\end{corollary}
\begin{proof}  
Let $\phi(u,v)=u_1^{-r}$ for example. Since $r<2$, we can choose conjugate exponents  $(p,q)$ such that $q\,r<2$. Due to the boundedness assumption on $\rho_0/\pi$ and $\pi$ being a probability measure, $\rho_0/\pi\in L^p(\pi)$. Further, $\phi\in L^q(\pi)$  as \[
\int \phi^q\, \pi = \frac{1}{Z_{\pi}} \int\frac{1}{u_1^{q\,r-1}} \,u_2\cdots u_d\, \exp\pp{-\beta\,\pp{\frac 12 \lambda\,\norm{(u,v)}^2 + G(u\odot v)}
}\,du\,dv\]
is finite for $q\,r-1<1$ as the exponential term decays rapidly at infinity due to the assumptions on $G$. Consequently, we may apply \cref{l1} to deduce $\mean{\phi((u,v)_t)}<\infty$ as required.

2.  By It\^o's formula applied to $(u,v)\mapsto u_1^\epsilon$ for $\epsilon>0$,
 \begin{gather}
     \begin{split}
\frac{1}{\beta}\,\epsilon^2\int_0^T {u_{1,t}^{\epsilon-2}}\,dt
&=
{u_{1,T}^\epsilon}-{u_{1,0}^\epsilon}+\lambda\,\epsilon\,\int_0^T {u_{1,t}^\epsilon}\,dt\\
&\quad+ \epsilon\int_0^T v_{1,t} \,[\nabla G(u_t\odot v_t)]_1 \,u_{1,t}^{\epsilon-1}\,dt-\sqrt\frac{2}{\beta} \,\epsilon\int_0^T u_{1,t}^{\epsilon-1}\,dW^{1,1}_t.
\end{split}\label{bla}\end{gather}
By Young's inequality with conjugate exponents $((2-\epsilon)/(1-\epsilon),(2-\epsilon))$, we have for any $\alpha>0$ that
\begin{align*} |v_1 \,[\nabla G(u\odot v)]_1 \,u_1^{\epsilon-1}|\le
\frac{1}{2-\epsilon} \alpha^{-(2-\epsilon)} |v_1 \,[\nabla G(u\odot v)]_1 |^{(2-\epsilon)}+
\frac{1-\epsilon}{2-\epsilon}\alpha^{(2-\epsilon)/(1-\epsilon)}\,u_1^{\epsilon-2}.
\end{align*}
We may set $\epsilon\,\alpha^{(2-\epsilon)/(1-\epsilon)}(1-\epsilon)/(2-\epsilon)=\epsilon^2/2\beta$ by choice of $\alpha$. Hence,  for a constant $c(\epsilon)$,
\begin{align*}
\epsilon\,\abs{\int_0^T v_{1,t} \,[\nabla G(u_t\odot v_t)]_1\,u_{1,t}^{\epsilon-1}\,dt} \le& 
c(\epsilon)\int_0^T  |v_{1,t} \,[\nabla G(u_t\odot v_t)]_1 |^{(2-\epsilon)}\,dt + \frac{1}{2\beta}\,\epsilon^2\int_0^T {u_{1,t}^{\epsilon-2}}\,dt.
\end{align*}
Substituting into \cref{bla}, we obtain
 \begin{align*}
\frac{1}{2\beta}\,\epsilon^2\int_0^T {u_{1,t}^{\epsilon-2}}\,dt
&\le
{u_{1,T}^\epsilon}-{u_{1,0}^\epsilon}+\lambda\,\epsilon\,\int_0^T {u_{1,t}^\epsilon}\,dt\\
&\quad+ c(\epsilon)\int_0^T |v_{1,t} \,[\nabla G(u_t\odot v_t)]_1|^{(2-\epsilon)}\,dt+\sqrt\frac{2}{\beta} \,\epsilon\,\abs{\int_0^T u_{1,t}^{\epsilon-1}\,dW^{1,1}_t}.
\end{align*}

On the right-hand side, $q$th moments are finite, which we now demonstrate for the last two terms. By the Burkholder--Gundy--Davis inequality \cite[Theorem 3.28]{Karatzas1991-xr} and Jensen's inequality for $1<q\le 2$, for a constant $c$,
\[
\mean{\abs{\int_0^T u_1^{\epsilon-1}\,dW_1(t)}^q}
\le c\,
\mean{\bp{\int_0^T {u_1^{2\epsilon-2}}\,dt}^{q/2}}
\le c\,
\bp{\mean{\int_0^T \frac{1}{u_1^{r}}\,dt}}^{q/2}
\]
with $r=2-2\epsilon$.
The integrand on the right-hand side is bounded for $\epsilon>0$ by part 1 with $r=2-2\epsilon<2$. 
Under the assumption on $G$, we have that $v_1 [\nabla G (u\odot v)]_1$ is polynomially growing and has bounded $q$th moments by \cref{wp}.  Hence, $q$th moments of the integral term $c(\epsilon)\int \cdots\,dt$ are finite.
   
   We conclude that $q$-moments of $\int_0^T {u_1^{\epsilon-2}}\,dt$ are finite for any $\epsilon>0$.
\end{proof}
We will also require the following time-regularity bounds. 
\begin{proposition} \label{timeregbnd} Let $(u,v)_t$ denote the solution of \cref{main} with initial distribution $\rho_0$. Suppose that the assumptions of \cref{cor} hold. For $1\le q< 2$, $T>0$, and $\epsilon\in(0,1/2)$,  there exists $c>0$ such that, for $0\le s,t\le T$,
\begin{enumerate}
    \item $\mathbb{E} \bp{ {\|(u,v)_t-(u,v)_s\|^q}\strutb}<c\, |t-s|^{q/2}$.
    \item $\mathbb{E} \bp{\frac {|u_{i,t}-u_{i,s}|}{u_{i,s}\,u_{i,t}}}\le c\, |t-s|^{1/2-\epsilon},\qquad i=1,\dots,d$.
\end{enumerate}
\end{proposition}
\begin{proof}   
 1. We treat the $u_1$ component only ($v$ is handled similarly). From \cref{main},
\[
u_{1,t}=u_{1,s}+\frac 1\beta\int_s^t \frac{1}{ u_{1,r}}\,dr + \cdots+ \sqrt\frac{2}\beta\int_s^t dW^1_r,
\]
where we omit several integrals that are straightforward to estimate under the moment bound in \cref{wp}. 

By Holder's inequality with conjugate exponents $(2-\epsilon, (2-\epsilon)/(1-\epsilon))$,
\[
\int_s^t \frac{1}{ u_{1,r}}\,dr 
\le |t-s|^{(1-\epsilon)/(2-\epsilon)}\,  {\bp{\int_s^t\frac{1}{|u_{1,r}|^{2-\epsilon}} \,dr}^{1/(2-\epsilon)}}.
\]
Now evaluate the $q$th moments and use moment bounds on Gaussian increments to find, for a constant $c_q>0$, that
\begin{align*}
(\mathbb E |u_{1,t}-u_{1,s}|^q)^{1/q}
&\le |t-s|^{(1-\epsilon)/(2-\epsilon)}\,\frac 1\beta\,  \pp{\mathbb{E} \bp{\int_s^t\frac{1}{|u_{1,r}|^{2-\epsilon}} \,dr}^{q/(2-\epsilon)}}^{1/q} + \cdots+\sqrt{\frac 2\beta}\,c_q\,|t-s|^{1/2}.
\end{align*}
By \cref{cor}, the $\mathbb{E} \bp{\pp{\int_s^t\frac{1}{|u_r|^{2-\epsilon}} \,dr}^q}$ is finite for $\epsilon>0$ and $1\le q< 2$. Hence, we have shown, for some constant $c$ depending on $\epsilon$ and $q$, that
$(\mathbb E \|u_t-u_s\|^q)^{1/q}\le c\,\abs{t-s}^{1/2-\epsilon}.
$

2. Let $\mathcal R=\{\min\{u_i(t),u_i(s)\}<\Delta\}$. Then, $\prob(\mathcal R) \le  \prob\pp{u_i(t)<\Delta} +  \prob\pp{u_i(s)<\Delta}$.  By Chebyshev's inequality, $\prob\pp{u_i(t)<\Delta}\le  \Delta^{2-\epsilon}\, \mean{u_i(t)^{\epsilon-2}}$ for $\epsilon>0$. Hence, using \cref{cor}, for a constant $c$, $\prob\pp{\mathcal R}\le c \,\Delta^{2-\epsilon}$. Now, split the expectation with respect to $\mathcal R$ as 
 \begin{align*}
     \mathbb{E} \bp{\frac {|u_i(t)-u_i(s)|}{u_i(s)\,u_i(t)}} 
     =&     \mathbb{E} \bp{\frac {|u_i(t)-u_i(s)|}{u_i(s)\,u_i(t)}1_{\mathcal R}}
     +    \mathbb{E} \bp{\frac {|u_i(t)-u_i(s)|}{u_i(s)\,u_i(t)} (1-1_{\mathcal R})}.
 \end{align*}
 We estimate each term separately.
 
 For conjugate exponents $(p,p')$, Holder's inequality applied to the first term provides
 \begin{align*}
     \mathbb{E} \bp{\frac {|u_i(t)-u_i(s)|}{u_i(s)\,u_i(t)}1_{\mathcal R}}
  \le&     \pp{\mathbb{E} \bp{\abs{\frac 1 {u_i(t)}-\frac1{u_i(s)} }^p}
  }^{1/p}\;\prob\pp{\mathcal R}^{ 1/{p'}}. 
 \end{align*}
 Choose $p'=2+\epsilon$. Then its conjugate exponent $p$ is less than $2$, so that $\mean{u_i^{-p}}<\infty$ as a result of Corollary \ref{cor}. As $\prob(\mathcal{R})\le c\, \Delta^{2-\epsilon}$ for some $c$, 
  the first term is bounded by $c\,\Delta^{
  (2-\epsilon)/(2+\epsilon)}= c\,\Delta^{1-2\epsilon/(2+\epsilon)}$. 

  For the second term, write
  \[
     \frac {|u_i(t)-u_i(s)|}{u_i(s)\,u_i(t)} (1-1_{\mathcal R})
   =    \frac {|u_i(t)-u_i(s)|^{1-\epsilon}}{|u_i(s)\,u_i(t)|^{1-\epsilon}}\,\frac {|u_i(t)-u_i(s)|^{\epsilon}}{|u_i(s)\,u_i(t)|^{\epsilon}}\,
   (1-1_{\mathcal R}).
   \]
For $(p,p')$ conjugate (different to the ones above),
 \begin{align*}
&   \mathbb{E} \bp{\frac {|u_i(t)-u_i(s)|}{u_i(s)\,u_i(t)} (1-1_{\mathcal R})}\\
 &\qquad  \le \pp{  \mathbb{E} \bp{\frac {|u_i(t)-u_i(s)|^{(1-\epsilon)p}}{|u_i(s)\,u_i(t)|^{(1-\epsilon)p}}}}^{1/p}\,\pp{ \mathbb{E} \bp{\frac {|u_i(t)-u_i(s)|^{\epsilon\,p'}}{|u_i(s)\,u_i(t)|^{\epsilon\,p'}}(1-1_{\mathcal R})}}^{1/p'}\\
&\qquad     \le  \pp{\mathbb{E} \bp{\frac {|u_i(t)-u_i(s)|^{(1-\epsilon)\,p}}{|u_i(s)\,u_i(t)|^{(1-\epsilon)\,p}}}}^{1/p}\, \Delta^{-2\epsilon
     }\,\pp{\mathbb{E} \bp{ {|u_i(t)-u_i(s)|^{\epsilon\, p'}}}}^{1/p'}.
   \end{align*}
In the last line, we exploit  $u_i(s),u_i(t)>\Delta$ outside $\mathcal R$.
The third factor is $O(|t-s|^{\epsilon/2})$ by part 1 if $1<\epsilon p'<2$. Choose $p'=3/2\epsilon$. Then $p=1/(1-1/p')=1/(1-(2/3)\epsilon)$.

Focus now on the first factor: for $(r,r')$ conjugate,
\begin{align*}
&\pp{\mathbb{E} \bp{\frac {|u_i(t)-u_i(s)|^{(1-\epsilon)\,p}}{|u_i(s)\,u_i(t)|^{(1-\epsilon)\,p}}}}\\
&\le \pp{\mathbb{E} \bp{ {|u_i(t)-u_i(s)|^{(1-\epsilon)\,p\,r'}}}}^{1/r'}\, \pp{\mathbb{E} \bp{\frac {1}{|u_i(s)\,u_i(t)|^{(1-\epsilon)\,p \,r}}}}^{1/r}.
\end{align*}
This is a product of two terms. The first term is $O(|t-s|^{(1-\epsilon)\,p/2})$ by part 1 with  $q=(1-\epsilon)\,p\, r'$.
   The second term is finite if $(1-\epsilon) \,p\, r<1$ (by \cref{cor} and $2\,a\,b\le a^2+b^2$).
We can achieve $(1-\epsilon)\,p\,r<1$ for any $\epsilon>0$, by choosing $1<r<1/(1-\epsilon)p=(1-(2/3)\epsilon)/(1-\epsilon)$. The resulting estimate is
 \begin{align*}
  \mathbb{E} \bp{\frac {|u_i(t)-u_i(s)|}{u_i(s)\,u_i(t)} (1-1_{\mathcal R})}
 &\le c\, \abs{t-s}^{(1-\epsilon)/2}\, \Delta ^{-\epsilon}\, \abs{t-s}^{\epsilon/2}
 = c\,|t-s|^{1/2}\, \Delta^{-\epsilon}.
 \end{align*}

    We have shown that
       \[\mathbb{E} \bp{\frac {|u_i(t)-u_i(s)|}{u_i(s)\,u_i(t)}} 
     \le  c \,\Delta^{1-\epsilon} +c\,|t-s|^{1/2} \,\Delta^{-\epsilon}.
     \]
     The right-hand side behaves like $\Delta^{1-2\epsilon/(2+\epsilon)} + \abs{t-s}^{1/2}\Delta^{-2\epsilon}$. Taking $\Delta = \abs{t-s}^\alpha$ and equating powers of $\abs{t-s}$ leads to $\alpha\pa{1-2\epsilon/(2+\epsilon)} = \frac12 - 2\alpha \epsilon$. This implies an upper bound of $c \abs{t-s}^{\frac12 - \frac{\epsilon}{2}}$ provided that $\epsilon \leq \frac12$. 
\end{proof}

\subsection{Links to the OU process and CIR process}
We highlight a link to some well-known SDEs, namely, the \emph{Cox--Ingersoll--Ross (CIR) process} and \emph{Ornstein--Uhlenbeck process} \citep{feller1951two}. Girsanov's theorem allows us to transfer properties of these SDEs to our dynamics and therefore provides an alternative way of establishing well-posedness and strong Feller properties. This link will also serve as inspiration for the development of numerical schemes in \cref{sec:disc}.

When the data term $G\equiv 0$, our Langevin dynamics decouples into some well-known SDEs for modeling interest rates. In particular, the \emph{Cox--Ingersoll--Ross (CIR) process} is
\begin{equation}\label{eq:cir1}
    dy_t = \kappa\,(\mu-y_t) \,dt + \theta \sqrt{y_t} \,dW_t.
\end{equation}
This is known to have  unique solutions with $ \mathbb{P}(y_t>0, \forall t\geq 0) = 1$ if and only if the Feller test holds, which occurs for $2\kappa\mu\geq \theta^2$.
In this case, by defining $\hat u_t = \sqrt{y_t}$ (this is the so-called Lamperti transformation), It\^o's formula implies that
\begin{equation}\label{eq:cir_2}
    d\hat u_t = \frac{(4\kappa \mu-\theta^2)/8}{\hat u_t}\,dt -\frac{\kappa}{2}\,\hat u_t \,dt + \frac{\theta}{2} \,dW_t.
\end{equation}
This is precisely our $u$ process when $G\equiv 0$, with $\theta = 2\sqrt{2/\beta}$, $\kappa = 2\lambda$ and $\mu = 2/(\beta\lambda)$. Note that this corresponds to the edge case of the Feller test as $2\kappa\mu = \theta^2$ \citep{feller1951two,carlo2001monte}.

The \emph{Ornstein–-Uhlenbeck process}  is given by 
$
d\hat v_t =-\lambda\, \hat v_t\, dt + \sqrt{\frac{2}{\beta}}\,dW_t^2.
$
This is precisely our $v$ process with $G\equiv 0$.

Given the connections mentioned above,
the general idea for establishing well-posedness of our SDE is to interpret \eqref{main} as  `perturbed' CIR and OU processes. Girsanov transformations will allow us to absorb the data term into the Brownian motion and hence inherit the required properties from CIR and OU processes.  The results shown here can also be derived using the (simpler) arguments involving the cylindrical transformation \cref{cyl}; however, as a bonus, the arguments here additionally show that $u(t)$ is positive almost surely.

In the following, let $\hat z_t \eqdef (\hat u,\hat v)_t$ satisfy
\begin{equation}\label{potential_free_sde_z}
d \hat z_t = \biggl[-\lambda z_t + \binom{1/(\beta \hat u_t)}{0} \biggr]\,dt + \sqrt{\frac{2}{\beta}} \,d  W_t,
\end{equation}
with given initial data $\hat z_0$. The first component $\hat u_t$ of $\hat z_t$ is a CIR process. The CIR process is positive almost surely as, following \cite{Dereich2011-tp}, each component has the form $d\hat u_t = ((4\kappa \mu - \theta^2)/(8\hat u_t))\,dt - \frac12 \kappa \hat u_t\, dt + (\theta/2)\,d\hat W_t$ for $\theta^2 = 8/\beta$, $\kappa=2\lambda$, and $\mu=2/(\beta \lambda)$, which satisfies the Feller positivity condition $2\kappa\mu\geq \theta^2$ with equality.

In order to apply  Girsanov transformations, we first establish the Novikov condition.
\begin{lemma}[Novikov condition]\label{lem:novikov}
    Let \cref{reg1} hold and $b(u,v) = -\binom{v\odot \nabla G(u\odot v)}{u\odot \nabla G(u\odot v)}$. For all $T>0$ and $\gamma>0$, we have
    $$
    \EE\biggl[ \exp\biggl(\gamma \int_0^T \norm{b(\hat z_s)}^2 \,ds\biggr) \biggr] <\infty.
    $$
    There exists a weak solution $(u_t,v_t)$ of \cref{main} that is unique in law. Furthermore, $u_t>0$ almost surely.
\end{lemma}
\begin{proof}
By \cite[Corollary 5.14]{Karatzas1991-xr}, it is enough to find an $\epsilon>0$ such that, for all $t\in [0,T]$, 
$
\EE\left[ \exp\pa{\int_t^{t+\epsilon}  \gamma \norm{b(\hat z_s)}^2} ds \right] <\infty.
$
Since $\cinfty = \norm{\nabla G}_\infty<\infty$, we obtain
\begin{align*}
    \EE\left[ \exp\pa{\gamma \int_t^{t+\epsilon} \norm{b(\hat z_s)}^2 \,ds} \right] \leq  \EE\left[\exp\pa{ \gamma \,\cinfty^2\int_t^{t+\epsilon} \norm{\hat z_s}^2 \,ds }\right].
\end{align*}
By Jensen's inequality,
\begin{equation}\label{eq:novikov}
    \EE\left[\exp\pa{ \int_t^{t+\epsilon} \gamma\, \cinfty^2\, \norm{\hat z_s}^2 \,ds} \right] \leq  \frac{1}{\epsilon}  \int_t^{t+\epsilon} \EE\Bigl[ \exp\pa{\epsilon\,\gamma\, \cinfty^2\,   \norm{\hat z_s}^2 } \Bigr] \,ds.
\end{equation}

We now apply It\^o's lemma to bound $\EE \bigl[ \exp(\epsilon\, \gamma\, \cinfty^2 \norm{\hat z_s}^2 ) \bigr]$. 
Given  $\alpha>0$ and  $\hat y_t \eqdef\exp\left(\frac{\alpha}{2}\norm{( \hat u_t,\hat v_t)}^2 \right)$, 
we get
\begin{align*}
    d\hat y_t &= \pa{-\lambda \alpha  + \frac{\alpha^2}{\beta}} \,\norm{\hat z_t}^2y_t\, dt + \pa{\frac{3\alpha d}{\beta}}\, \hat y_t\, dt + \sqrt{\frac{2}{\beta}} \alpha \hat y_t  \pa{\hat u_t^\top d W^1_t + v_t^\top d W^2_t}.
\end{align*}
Taking expectation,
$$
\frac{d}{dt}\EE\bigl[\hat y_t\bigr] \leq \EE\left[\pa{-\lambda \alpha + \frac{\alpha^2}{\beta}}\norm{\hat z_t}^2 \hat y_t + \pa{\frac{3\alpha d}{\beta}} \hat y_t\right].
$$
Provided that $ \alpha<\beta \lambda $, we have $\frac{d}{dt}\EE\bigl[\hat y_t\bigr] \leq  \pa{\frac{3\alpha d}{\beta}} \EE\bigl[\hat y_t\bigr]$ and therefore
$
\EE\bigl[y_t\bigr]  \leq \hat y_0 \exp\biggl(  \pa{\frac{3 \alpha d}{\beta}} t\biggr),
$
where $\hat y_0=\exp(\alpha \norm{z_0}^2/2)<\infty$. 
Letting $\alpha = 2\epsilon \gamma \cinfty^2 $, the condition on $\alpha$ is satisfied if $\epsilon  <\beta\lambda/(2\gamma \cinfty^2)$. 
Substituting for $\alpha$,
$
\EE\bigl[\hat y_t\bigr]  \leq y_0 \exp\biggl(2\epsilon \gamma \cinfty^2  \pa{\frac{3 d}{\beta}} t\biggr).
$
Returning to \eqref{eq:novikov}, for all $t\leq T-\epsilon$, we see
$$
\EE\left[ \exp\pa{\int_t^{t+\epsilon} \gamma \cinfty^2 \norm{\hat z_s}^2} ds \right] \leq \hat y_0\exp\biggl( 2\epsilon \gamma \cinfty^2  \pa{\frac{3 d}{\beta}} T\biggr).
$$
Thus, the Novikov condition is satisfied.

We can therefore apply Girsanov's theorem. Then, a solution of \cref{main} exists and the solution is unique in law~\citep{Karatzas1991-xr}. Further, for all $T>0$, with the martingale defined by $M_T = \exp(\int_0^T b(\hat z_t) \cdot dW_t - \frac12 \int_0^T b(\hat z_t)^2 dt)$, we see $
\PP(\exists t\in [0,T], \; u_t =0)=
\EE\Bigl[1_{\{\exists t\in [0,T], \; \hat u_t=0\}} M_T\Bigr]$.
Since the CIR process $\hat u_t$ is strictly positive almost surely,  $\PP(\exists t\in [0,T], \;  u_t=0) = 0$.
By the Borel--Cantelli lemma, there are only finitely many $T$ for which the event $\{\exists t\in [0,T],\; u_t=0\}$ occurs so $\PP(u_t>0, t>0) = 1$. 
\end{proof}

\begin{proposition}
Let \cref{reg1} hold.  Then,  $(u,v)_t$ is  strong Feller.
\end{proposition}
\begin{proof}
If $b(x_s) \equiv 0$, then the processes $(u,v)_t$ decouple component-wise into CIR and OU processes.   Sufficient conditions are provided by  \cite[Theorem 2.1]{maslowski2000probabilistic} to guarantee that Girsanov transformed processes are strong Feller.  We check  the conditions of their theorem: (i) is satisfied since each of these components is strong Feller with  known densities. Condition (ii) is due to the Novikov condition proven above. Condition (iii) holds due to   \cite[Remark 2.3]{maslowski2000probabilistic} along with Lemma \ref{lem:novikov}. 
\end{proof}

\subsection{Link to Riemannian Langevin diffusion}

When sampling distributions on non-Euclidean spaces,
Riemannian Langevin dynamics \citep{girolami2011riemann} are typically more efficient than standard Langevin dynamics since they scale the gradient to account for the geometry of the manifold. In general, if sampling from $\pi(\mathbf z) \propto \exp(-H(\mathbf z))$ on some manifold $\mathcal{M}$ with metric $\mathfrak{g}$, the Riemannian Langevin dynamics  are \cite{xifara2014langevin}
\begin{equation}\label{eq:riemannian}
    d\mathbf z_t = -\mathfrak{g}(\mathbf z_t)^{-1} \nabla H(\mathbf z_t)\, dt + \Gamma(\mathbf z_t)\,dt + \sqrt{2} \mathfrak{g}(\mathbf z_t)^{-1/2} dW_t,
\end{equation}
where $\Gamma_j(\mathbf z) = \sum_k \partial_k [ (\mathfrak{g}(\mathbf z)^{-1})_{jk} ]$.
The last two terms correspond to Brownian motion on the manifold $(\mathcal{M},\mathfrak{g})$ with the correction drift term taking into account the curvature of the manifold. The special setting where the metric is the Hessian of some convex Legendre-type function is called `Mirror-Langevin'. For example,  in the case of sampling probability densities restricted to the positive real plane $\RR^d_+$ or the simplex, a natural choice is $\mathfrak{g}(\mathbf{z}) = \nabla^2 \mathcal H(z)$ where $\mathcal H$ is the Boltzmann--Shannon entropy $\mathcal H(\mathbf z) = \sum_i z_i\log(z_i) - z_i$ \citep{zhang2020wasserstein}.

Recall from \eqref{eq:gibbs} for $\beta=1$ that $\rho(x)$ is the marginal distribution to the following distribution on $\mathbf{z}= (\eta,x) \in \mathcal X=  \real_+^d\times \real^d$:  $\tilde \rho(\eta, x) \propto \exp(-H(\eta, x))$ where  $$
H(\eta,x) \eqdef G(x) + \lambda \sum_{i=1}^d \frac{x_i^2}{8\eta_i} + 2\lambda \eta_i + \frac12 \log(\eta_i).
$$
In the following proposition, we see that our Langevin dynamics \eqref{main} on $(u,v)$ corresponds to Riemannian Langevin dynamics on $\vec z$ with metric $\mathfrak{g}(\vec z)\eqdef M(\vec z)^{-1}$, with $M(\vec z)$ as given in the proposition below. The proof is a direct consequence of It\^o's formula and given in Appendix \ref{app-sec-riemann-geom}.

\begin{proposition}\label{prop-riemann-geom}
    Let $(u,v)_t$ follow the SDE \eqref{main} with $\beta =1$. Then, letting $\vec z_t=(\eta, x)_t$ with $\eta = \frac14 u^2 $ and $x =  u\odot v$, we have that $\vec z_t$ follows \eqref{eq:riemannian} with inverse metric
    $$
    \mathfrak{g}(\vec z)^{-1} = M(\vec z)\eqdef \begin{pmatrix}
 \diag(\eta)&\diag( x)/2 \\
\diag(x)/2&\diag( x^2/(4\eta)) + 4\diag(\eta)  
\end{pmatrix}.
    $$
The distribution $\rho(t,\cdot)$ of $\mathbf{z}(t)$ solves the Fokker--Planck equation
\begin{equation}
\frac{\partial}{\partial t} \rho
= \mathrm{div}\pa{{\strut}M(\mathbf{z}) \nabla
{
\pa{{\strut}H(\mathbf{z})+ \log(\rho)}
} 
\rho }.
\end{equation}
\end{proposition}
\begin{remark}
 By direct computations, one can check that 
for $\phi(\mathbf{z}) = \frac18 \sum_{i=1}^d\frac{\abs{x_i}^2 }{\eta_i}+2\eta_i +4 \sum_{i=1}^d \mathcal H(\eta_i),$
where $\mathcal H(\eta) =
\pa{\eta\log(\eta)-\eta}$, the corresponding metric is 
$$
\mathfrak{g}(\mathbf{z}) = M(\mathbf{z})^{-1} = \begin{pmatrix}
   \frac12 \Id_d &0_{d\times d}\\
    0_{d\times d} &  \Id_d
\end{pmatrix} \nabla^2 \phi(z)  \begin{pmatrix}
   \frac12\Id_d &0_{d\times d}\\
    0_{d\times d} &  \Id_d
\end{pmatrix}.
$$
\end{remark}

\subsection{Ergodicity}\label{erg}
Let $\mathcal L$ denote the infinitesimal generator of \cref{main} and $\rho(t,\cdot)$ denote the distribution of $(u,v)_t$.
Then $(e^{\mathcal L  t}\phi)(X)=\mean{\phi((u,v)
_t) \mid (u,v)_0=X}$ for measurable functions $\phi\colon \mathcal X\to\real$. We seek to show 
that $e^{\mathcal L \,t }\phi$ converges to $\pi(\phi)=\int \phi(x)\,d\pi(x)$ or, alternatively, $\rho_t$ converges to $\pi$. 

If the data term $G$ is uniformly bounded, then one could invoke standard theories (Bakry-Emery criterion) to deduce ergodicity; see Appendix \ref{app-sec-ergod}. This is, however, a strong assumption. In the following, we prove a different ergodicity result that assumes $x^\top \nabla G(x)$ is bounded below and which holds for many unbounded examples, including $G(x)=\norm{Ax-y}^2$.
The proof depends on the theory of Lyapunov--Foster drift conditions~\citep{Meyn2012-oq,Lelievre2016-hx}. 

For the Lyapunov function $V(u,v)=1+\norm{(u,v)}^2$, let $L^\infty_V$ denote the Banach space  of  measurable functions $\phi\colon \mathcal X\to\real$ such that the norm $\norm{\phi}^\infty_
V\coloneq\sup_x \abs{\phi(x)/V(x)}<\infty.$
\begin{theorem} \label{erg2} Under \cref{reg1}, the solution $(u,v)_t$ of \cref{main} converges to the invariant $\pi$ geometrically. In particular, for $V(u,v)=1+\norm{(u,v)}^2$ and some $C,\rho>0$,
    \[\snorm{e^{\mathcal L\, t}\phi   - \pi(\phi)}^\infty_V\le  C \exp(-\rho\, t) \norm{\phi-\pi(\phi)}^\infty_V,\qquad t>0,\,\phi\in L^\infty_V.
    \]
\end{theorem}
\begin{proof} In discrete time, this follows \cite[Theorem 2.25]{Lelievre2016-hx} if  a Lyapunov condition holds for $V$  and a minorization condition holds for the transition density.  The extension to continuous time  \cite[Equation 2.77]{Lelievre2016-hx} is achieved by writing $t=k \tstep +s$ to create a  Markov chain on discrete time $k$ and correction factor depending on $s\in(0,\tstep)$.

The process $(u,v)_t$ is strong Feller and irreducible by \cref{wp}. Hence,  the minorization conditions hold for any compact set~\cite[Lemma 2.3]{mattingly2002ergodicity}. 

 Under \cref{reg1},  $x^\top \nabla G(x) \ge -\ccc $ for some constant $\ccc \in \mathbb{R}$.   We have, from \cref{main},
    \begin{align*}
    \mathcal L V(u,v) 
    &= -2\,\lambda\,\|(u,v)\|^2  - 4(u \odot v)^\top \nabla G(u \odot v) + 2d+4\,d/\beta
    \leq -2\,\lambda\,\|(u,v)\|^2  +4\ccc+ 2d+4\,d/\beta.\end{align*}
    Let $c$ be a constant, such that $0 < c < 2\lambda$ and $b=4\ccc+2d+4d/\beta  + c > 0$. We  have $\mathcal L V(u,v) \leq -c\,V(u,v) + b$ as required. 
    \end{proof}

\section{Discretization}\label{sec:disc}

\rev{We consider the following implicit-explicit time-stepping approximation: for time step $\deltat$, we seek an approximation $(u,v)^k$ to $(u,v)_{k \tstep}$ via
\begin{gather}
\begin{split}
u^{k+1}-u^k&=-\lambda\,   u^{k+1}\,\deltat - v^k \odot\nabla G(u^k \odot v^k)\,
\deltat
 +\frac{1}{\beta\, u^{k+1}}\,\deltat+ \sqrt{\tfrac{2}{\beta}} \,\deltaw^1_{k},\\
v^{k+1}-v^k&=-\lambda\,  v^{k+1}\,\deltat -  u^k\odot \nabla G(u^k \odot v^k)\,\deltat+  \sqrt{\tfrac{2}{\beta}}\,\deltaw^2_k,
\label{discrete_scheme}
\end{split}
\end{gather}
where $\deltaw^i_k=W^i_{t_{k+1}}-W^i_{t_k}$. This scheme is inspired by the numerical scheme of  ~\cite{Dereich2011-tp} for the CIR process (corresponding to the $u$ update with  $G\equiv 0$). It is explicit in the potential term $G$; the implicitness in the regularisation terms ensures positivity in $u$
and is straightforward to implement 
 since   $$ u^{k+1} = \frac{\tilde{u}^k + \sqrt{\tilde{u}^k\odot \tilde{u}^k + 4\deltat(1+\deltat\lambda)/\beta}}{2(1+\deltat\lambda)}$$ with $\tilde u^k =u^k - \deltat\, v^k \odot \nabla G(u^k\odot v^k) + \sqrt{\tfrac{2}{\beta}}\,\deltaw^1_k$.


For convenience, define the explicit part of the drift term as $f$, so that 
\begin{equation}\label{eq:f}
    f(u,v) \eqdef \binom{f_1(u,v)}{f_2(u,v)} \eqdef \binom{v\odot \nabla G(u\odot v)}{u\odot \nabla G(u\odot v)}.
\end{equation}

\begin{remark}[Relation to proximal gradient Langevin]\label{rem:prox}
By defining $\Psi(u,v) \eqdef \frac{\lambda}{2} \norm{(u,v)}^2 -\beta \log(u) $,  $(u, v)^{k+1} = \Pi_{\deltat}(z^{k+1})$,
one can write \eqref{discrete_scheme}  as 
    \begin{align}\label{PIPI}
    z^{k+1} = \Pi_{\deltat}(z^{k}) - \deltat\,  f(\Pi_{\deltat}(z^{k})) + \sqrt{{2/\beta}}\,\Delta W_k,
\end{align}
where $\Pi_{\deltat}(z) = \mathrm{Prox}_{\deltat \Psi}(z) = \mathrm{argmin}_{x} (\deltat \,\Psi(x) + \frac12 \norm{x-z})$ is a non-expansive operator.
This corresponds to an application of the proximal gradient Langevin (PGLD) algorithm introduced in \cite{durmus2019analysis}. Note however that their analysis does not extend to our setting since $f$ is non-convex and lacks global Lipschitz properties. 
\end{remark}
We  establish convergence  of our discretized scheme on a fixed time interval $[0,T]$ under the following assumptions. 
The assumption on the initial distribution $\rho_0$ is satisfied for example if $\rho_0$ has a continuous density with compact support.}
\rev{\begin{assumption}\label{reg3} 
\begin{itemize}
    \item[(i)] $G$ satisfies \cref{reg1} and $\nabla G$ is $\clip$-Lipschitz.
    \item[(ii)] $(u^0,v^0)$ and $(u(0),v(0))$ follow an initial distribution $\rho_0\ll \pi $ such that $\frac{\rho_0}{\pi}$ is uniformly bounded.
\end{itemize}
\end{assumption}
}

For our analysis, we will frequently use the local Lipschitz property of $f$.
\begin{lemma}\label[lemma]{Lip_f}
      Under \cref{reg3}, 
      the function $f$ defined in \eqref{eq:f} is locally Lipschitz on the ball of radius $R$ with Lipschitz constant $C_R = 2\,L\,R^2  + \cinfty$:
$
\norm{f(u,v) - f(u',v')} \leq  C_R \norm{(u-u', v-v')}
$
for all $(u,v),(u',v')\in \mathcal X$ such that $\norm{(u,v)}, \norm{(u',v')} \leq R$.
  \end{lemma}
  \begin{proof}  Under \cref{reg3}, $\nabla G$ is $L$-Lipschitz, 
   $B \eqdef \sup_{x} \norm{\nabla G(x)}<\infty$, 
      Hence, for $v,\tilde{v}$ such that $\|v\|\leq R,\|\tilde{v}\|\leq R$,  $\norm{v \odot\nabla G(u\odot v) -\tilde v \odot\nabla G(\tilde u\odot \tilde v)}$ is upper bounded by
\begin{align*}
 \cinfty \norm{\tilde v - v} + \norm{v} \norm{\nabla G(u\odot v) - \nabla G(\tilde u\odot \tilde v)}& \leq \pa{\cinfty + 2 \clip R^2} \Bigl(\norm{v-\tilde v} \vee\norm{u-\tilde u}\Bigr).\qedhere
\end{align*}
    \end{proof}

\subsection{Convergence to stationary distribution}

We show in \cref{NGE} that  the numerical approximation $(u,v)^k$ indeed has an invariant distribution. 

We first introduce some notation.
For $n>0$, $x\in \mathcal X$, and Borel sets $A\subset \mathcal X$, let $P^n(\mu_0,A) \eqdef \PP(X^n \in A \mid X^0 \sim \mu_0)$ for initial distribution $\mu_0$. For $\phi\colon \mathcal X\to \real$, 
 $$(P^n\phi)(x) \eqdef \int_{\mathcal X} \phi(y)P^n(\delta_x,dy) = \EE[\phi(X^n) \mid X^0 = x].$$
Given $\BETA>0$,  define $\norm{\phi}_\BETA \eqdef \sup_{x\in\mathcal X} \frac{\abs{\phi(x)}}{1+\BETA \,V(x)}$ and
$$
d_\BETA(x,y) \eqdef \begin{cases}
    0, &x=y,\\
    2+\BETA\, V(x)  + \BETA\, V(y), &x\neq y.
\end{cases}
$$
The argument is based on showing  $P^n$ is a contraction with respect to
$$
d_\BETA(\mu_1,\mu_2)\coloneq \sup_{\threenorm{\phi}_\BETA=1} \int_\mathcal X\phi(x)\,d(\mu_1-\mu_2)(x), \quad\text{for measures $\mu_1,\mu_2$}
$$
with $\threenorm{\phi}_\BETA\eqdef\sup_{x\ne y}\abs{\phi(x)-\phi(y)}/d_\BETA(x,y)$. Following \cite{hairer2011harris}, $d_\BETA$ is equivalent to the weighted TV norm  $\norm{\mu}_\BETA \eqdef \sup_{\norm{\phi}_\BETA\leq 1}\abs{\int_{\mathcal X} \phi(x) \,d\mu(x)}$.

\begin{theorem}\label{NGE}
 
There exists $\bar \alpha \in(0,1)$, $t_0>0$ and $\BETA>0$ such that, for any $\phi\colon \mathcal X\to\real$ with $\norm{\phi}_\BETA \leq 1$,
 \begin{equation}\label{contraction}     
 \abs{P^{\lceil t_0/\deltat\rceil}\phi(x)-P^{\lceil t_0/\deltat\rceil}\phi(y)}\leq \bar \alpha \,d_\BETA(x,y),
 \end{equation}
 holds for all $\deltat$ sufficiently small.
 As a consequence,  there exists an invariant measure $\pi_{\tstep}$ such that,  for any initial distribution $\mu_0$ and any $T>0$,
 \begin{equation}\label{finale}
 \norm{  P^{\lceil T/\deltat\rceil}(\mu_0,\cdot)- \pi_{\deltat}}_\BETA \leq \bar \rho^T \norm{\mu_0-\pi_\tstep}_\BETA,
 \end{equation}
 where  $\rho = \bar \alpha^{1/t_0}$ is independent of $\deltat$.
\end{theorem}

This is a consequence of minorisation and Lyapunov drift conditions, given  in the following lemmas.

\begin{lemma}[drift condition]\label{lem:2ndmoments_discrete}
Let $V(u,v)=1+\|(u,v)\|^2$.
Assume that $\deltat  < \min\{\lambda/\cinfty^2,1/4\lambda\}$, where $\cinfty=\|\nabla G\|_\infty$. Then, the numerical approximation $(u^k,v^k)$ defined by \cref{discrete_scheme} satisfies
    $$
   \EE \bp{ V((u,v)^{k+1})  \mid (u,v)^k}\leq \pp{1-\frac12 \lambda\,\tstep}\,  V((u,v)^k)  +C \,\tstep,
    $$
    where $C\eqdef\frac 12 \lambda +4(8+K\beta+d)/\beta$.

    Let $V_4(u,v) = 1+ \norm{(u,v)}^4$. There exists $C,c,\nu>0$ such that for all $\deltat<c\lambda$,
    $$
    \EE[V_4((u,v)^{k+1})\mid (u,v)^k] \leq (1-\nu\,\tstep) \,V_4((u,v)^k) + C\, \deltat.
    $$
\end{lemma}

\begin{proof} We prove the drift condition for the quadratic functional $V$; the calculation for the quartic $V_4$ is provided in the Appendix \ref{supp_drift}. We have
    \begin{align}
        &(1+ \lambda \deltat) u^{k+1} - \frac{\deltat}{\beta u^{k+1}} = u^k -  v^k \odot \nabla G(u^k\odot v^k)\, \deltat + \sqrt{\frac{2}{\beta}} \,\deltaw^1_k, \label{eq:drift_comp1}\\
         &(1+ \lambda \deltat) v^{k+1} = v^k -  u^k \odot \nabla G(u^k \odot v^k)\,\deltat + \sqrt{\frac{2}{\beta}} \,\deltaw^2_k.
    \end{align}
    We square the norm and expectation of the above, use the independence between the increments $\deltaw^i_k$ and the information at time $t_k$, and rely on the trivial fact that \rev{$\|(u^{k+1})^{-1} \odot (u^{k+1})\|^2 = d$}. 
    As a result, conditional on step $k$, we obtain 
    \begin{align*}
        &(1+ \lambda \,\deltat)^2 \EE \bp{\snorm{u^{k+1}}^2\mid (u,v)^k} - 2\deltat (1+\lambda\, \deltat)\beta^{-1} d  \leq 
        {\frac{2}{\beta}}\, \EE \snorm{\deltaw^1_k}^2 
          + \snorm{u^k - \deltat \,v^k \odot \nabla G(u^k \odot v^k) }^2, \\
         &(1+ \lambda \,\deltat)^2 \EE\bp{\snorm{v^{k+1}}^2\mid (u,v)^k}  = \frac{2}{\beta} \,\EE\snorm{ \deltaw^2_k}^2  + \snorm{v^k - \deltat\, u^k \odot \nabla G(u^k\odot v^k)}^2.
    \end{align*}
Using the assumption $x^\top \nabla G(x) \geq -\ccc$ and $\cinfty=\norm{\nabla G}_\infty$, we obtain 
    \begin{align*}
        \snorm{u^k-\deltat \,v^k \odot \nabla G(u^k\odot  v^k)}^2 &\leq \snorm{u^k}^2  + \deltat^2 \cinfty^2 \snorm{v^k}^2 - 2 \,\Delta t\,(u^k \odot v^k)^\top \nabla G(u^k \odot v^k)\\
        &\leq \snorm{u^k}^2  + \deltat^2 \cinfty^2 \snorm{v^k}^2 +2\,\ccc\,\deltat.
    \end{align*}
   A similar statement holds for $\norm{v^k-\deltat\, u^k \odot \nabla G(u^k \odot v^k)}^2$.
    Combining this with $ \EE\norm{ \deltaw^i_k}^2= d\,\Delta t$ , we obtain  
    \begin{align*}
    (1+\lambda \,\deltat)^2 \EE\Bigl[\snorm{(u,v)^{k+1}}^2 \bigm\lvert (u,v)^k\Bigr] \leq (1+\deltat^2 \cinfty^2)\snorm{(u,v)^{k}}^2  + 
    4{ \frac{d\,( 1 + \deltat \,\lambda) + \ccc\,\beta + d}{\beta}}\,\deltat.
    \end{align*}
    Note that
    $
    \frac{(1+\deltat^2 \cinfty^2)}{(1+\lambda \deltat)^2} \leq \frac{(1+\deltat^2 \cinfty^2)}{1+2\lambda \deltat} <1-\frac12\lambda\, \tstep
    $
    provided that $0<\deltat  < \min\{\lambda/\cinfty^2,1/4\lambda\} $. Rewrite this in terms of $V(u,v)=1+\norm{(u,v)}^2$:
        \begin{align*}
        \EE\Bigl[V((u,v)^{k+1}) \Bigm
        \lvert (u,v)^k\Bigr] 
        \leq \pp{1-\frac 12 \lambda \tstep} V((u,v)^{k}) + 
   \pp{\frac 12 \lambda+ 4\frac{4\,(1+\deltat \lambda)+K\,\beta+ d}{\beta (1+\lambda \deltat)^2  }}\,\tstep.
    \end{align*}
Using the property $\tstep\,\lambda\le 1$, we can replace the last term by $C\tstep$ with $C\eqdef\frac 12 \lambda +4(8+K\beta+d)/\beta$.
\end{proof}
By first establishing \cref{escape,coupling}, we prove the coupling condition in \cref{minorisation} for the numerical method.
\begin{lemma}[escape probability]\label{escape}
Fix $0<r<R$. Let  $k_R = \inf\enscond{k}{\snorm{(u,v)^{k}}>R}$ and assume that $\snorm{(u,v)^0} \leq  r$. For some constants $c,C>0$, we have, for  $n\deltat = T$ \rev{ and $\deltat \leq c$},  
\begin{align*}
    \PP(k_R\leq n)
    \leq
    \exp\pa{-\frac{R^2}{\rev{2CT}}}\exp\pa{ \frac{r^2}{\rev{CT}} +\rev{2} }. 
    \end{align*}
    The constant $C$ depends only on $ \beta$ and the coercivity constant of $G$.
\end{lemma}
\begin{proof}
The start of this proof is similar to the proof of Lemma \ref{lem:2ndmoments_discrete} but we retain the Brownian increments: taking the square-norm of \eqref{eq:drift_comp1}
and using $
    \norm{u^k - v^k \odot \nabla G(u^k \odot v^k) \deltat}^2 \leq \norm{u^k}^2 + \deltat^2 B^2 \norm{v^k}^2 + 2K \deltat$,
    we obtain
$$
(1+\lambda \deltat)^2\snorm{u^{k+1}}^2 \leq \snorm{u^k}^2 + \deltat^2 B^2 \snorm{v^k}^2 + 2K \deltat + Z^1_k + 2d\deltat/\beta + 2d\deltat (1+\lambda \deltat) /\beta,
$$
where 
$Z^1_k\eqdef \frac{2}{\beta}
(\norm{\Delta W^1_k}^2 - d\deltat)+ \sqrt{\frac{8}{\beta}}\dotp{\Delta W^1_k}{u^k - v^k \odot \nabla G(u^k \odot v^k) \deltat}$.

A similar computation for $v^k$ with $Z^2_k\eqdef \frac{2}{\beta}(\norm{\Delta W^2_k}^2-d\deltat)+ \sqrt{\frac{8}{\beta}}\dotp{\Delta W^2_k}{v^k - u^k \odot \nabla G(u^k \odot v^k) \deltat}$ yields
$$
(1+\lambda \deltat)^2\snorm{v^{k+1}}^2 \leq \snorm{v^k}^2 + \deltat^2 B^2 \snorm{u^k}^2 + 2K \deltat + Z^2_k + 2d\deltat/\beta.
$$
We therefore have constants $c,C>0$ such that for all $\deltat\leq c$,
$$
(1+\lambda \deltat)^2 \snorm{(u,v)^{k+1}}^2 \leq (1+B^2 \deltat^2)\snorm{(u,v)^{k}}^2 + Z_k^1+Z_k^2 + C\deltat.
$$
Let $\xi_k \eqdef Z^1_k + Z^2_k$ and let $\mathcal{F}_n$  be the $\sigma$-algebra generated by $\Delta W^j_k$ for $j=1,2$ and $k=1,\dots,n-1$.
Note that $\EE[\xi_k\mid\mathcal{F}_k] = 0$ and $\EE[\xi_k^2\mid \mathcal{F}_k]\leq C ( \deltat^2 +(\snorm{u^{k}}^2 + \snorm{v^{k}}^2)\deltat)$ for some $C>0$.
Let $M_n\eqdef \sum_{k< n} \xi_k$. This is a martingale with quadratic variation $\langle M\rangle_n$ satisfying
\begin{equation}\label{eq:quad_var_bound}
    \langle M \rangle_n = \sum_{k<n} \EE\bp{\xi_{k}^2 \mid \mathcal{F}_k} \leq C \sum_{k<n} \pp{ \deltat^2 +(\norm{\smash{u^{k}
}}^2 + \norm{\smash{v^{k}}}^2)\,\deltat}.
\end{equation}
Using this notation, we can rewrite the above inequality as
\begin{equation}\label{eq:bis1}
          (1+\lambda\deltat)^2  \,\snorm{(u,v)^{k+1}}^2 \leq \pp{1+B^2 \deltat^2}\snorm{(u,v)^{k}}^2 + C\deltat + \xi_k - \frac{\alpha}{2}\EE[\xi_k^2\mid \mathcal{F}_k]  + \frac{\alpha}{2}\EE[\xi_k^2\mid\mathcal{F}_k].
\end{equation}

\textbf{Absorbing the variation term.}
 Equation \eqref{eq:bis1} can be written as
\begin{equation*}
           \snorm{(u,v)^{k+1}}^2 \leq \snorm{(u,v)^{k}}^2 + C\deltat + \xi_k - \frac{\alpha}{2}\,\EE[\xi_k^2\mid\mathcal{F}_k]  + \frac{\alpha}{2} \pp{\langle M\rangle_{k+1} - \langle M\rangle_{k}}.
\end{equation*}
Summing this over the $k$'s implies that
$$
\snorm{(u,v)^{k}}^2 - \frac{\alpha}{2} \langle M\rangle_{k} \leq \snorm{(u,v)^{0}}^2  + C\,k\,\deltat + \sum_{j< k}( \xi_j - \frac{\alpha}{2}\,\EE[\xi_j^2\mid\mathcal{F}_j]).
$$
\rev{
By moving $\frac{\alpha}{2} \langle M\rangle_{k} $ to the RHS and applying \eqref{eq:quad_var_bound}, it follows that,
\begin{align*}
    \snorm{(u,v)^{k}}^2 
    &\leq \snorm{(u,v)^{0}}^2  + C\,k\,\deltat + \pa{M_k - \frac{\alpha}{2}\,\langle M\rangle_k}  +\frac{\alpha C}{2} k \deltat^2  + \frac{\alpha C}{2} \sup_{j<k}\snorm{(u,v)^j}^2 k \deltat.
\end{align*}
So, for all $n\deltat \leq T$ and all $k\leq n$,
$$
\left(1-\frac{\alpha C T}{2}\right) \sup_{k\leq n}\snorm{(u,v)^{k}}^2  \leq \snorm{(u,v)^{0}}^2  + C(1+ \alpha \deltat/2)\,T + \sup_{k\leq n} \pa{M_k - \frac{\alpha}{2}\,\langle M\rangle_k}.
$$
}

\textbf{Estimate on escape probability.}
We now apply \cite[Lemma~A.1]{Lamba2007AdaptiveEulerMaruyama}:
Let \rev{$\alpha = 1/(C\,T)$}.
For $n\deltat= T$, we have
\begin{equation}\label{eq:prob_escap_pre}
\PP\pa{\frac12\sup_{k\leq n} \snorm{(u,v)^{k}}^2  \geq \snorm{(u,v)^{0}}^2  + \rev{2}CT + R} \leq \PP\pa{\sup_{k\leq n}\pp{M_k -  \frac{\alpha}{2}\langle M\rangle_k} \geq   R} \leq  \rev{e^{- R/(C\,T)}}.
\end{equation}

So, for $n\deltat = T$, as $k_R$ is the first escape time from the ball of radius $R$,
\begin{align*}
    \PP(k_R\leq n)
    &= \PP\pa{\sup_{k\leq n} \frac12 \snorm{(u,v)^k}^2  \geq \frac{R^2}{2}  }
    \leq 
    \exp\pa{-\frac{R^2}{\rev{2CT}} + \frac{\snorm{(u,v)^0}^2}{\rev{CT}} +\rev{2} }.\qedhere
\end{align*}
\end{proof}

The first coupling lemma considers the  coupling  event (i.e., $X^n=Y^n$) of two realisations $X^n, Y^n$ of the numerical method  for paths that do not leave the ball (i.e., $X^k, Y^k \in B(0,R)$ for $k=0,\dots,n-1$).
\begin{lemma}[coupling in the ball] \label{coupling} Consider initial data $x,y\in \mathcal X$ with $\norm{x},\norm{y}\le r$, some $r>0$. Consider  two solutions $(u,v)^n$ of the numerical method driven by different Brownian motions, denoted $X^n, Y^n$ with initial condition $X^0=x$, $Y^0=y$. For $R>0$ and $t_0>0$, let 
$E=\{X^1,\dots,X^{n-1},Y^1,\dots,Y^{n-1}\in B(0,R)\}$ and $\tilde E=E\cap \{X^n\ne Y^n\}$ where $n\,\Delta t=t_0$.

We may couple $X^n, Y^n$ such that the following property holds: if $t_0=1/C_R$ and $\Delta t\in(0,t_0]$ and $r\sqrt{ 3\,C_R}\ge 1$, then $\PP(\tilde E)\le 1-\tilde\delta$ with
 $\displaystyle
\tilde\delta
=  \frac{2}{\sqrt{2\pi}(1+r\sqrt{3\,C_R})}  \exp\left(-\frac{3\,r^2}{2} \, C_R\right)$.

\end{lemma}

\begin{proof}
The assumptions for  \cite[Theorem~4]{debortoli2020convergence} hold:  Assumption A1 holds due to the non-expansivity of $\Pi_\tstep$ in \cref{PIPI}; and Assumption A2($A$) holds as the drift  is  Lipschitz  on $B(0,R)$ with Lipschitz constant $L=C_R$ (\cref{Lip_f}). Hence, we may couple $X^n,Y^n$ so that $\PP(\tilde E)\le 1-\delta$ with 
\begin{align}\label{express_delta}
\delta= 2\inf_{x,y\in B(0,r)}\Phi\pp{-\|x-y\|/2\,\Xi_n^{1/2}(C_R)},
\end{align}
where $\Phi$ is the cumulative distribution function for the standard Gaussian, $C_R$ is the Lipschitz constant of the drift  on the set $B(0,R)$, and
$\displaystyle
\Xi_n(L)=\Delta t\sum_{k=1}^n (1+\Delta t\,L)^{-k}$.

For $\mu= 1/(1+\Delta t\,C_R)$, then $1-\mu=\Delta t\,C_R\, \mu$ and 
$$\Xi_n(C_R)= {\Delta t}\,\mu \frac{1-\mu^n}{1-\mu}= \frac{1-\mu^n}{C_R}\ge \frac{1-\exp(- t_0\, C_R/(1+\Delta t\,C_R))}{C_R},\qquad n\Delta t=t_0,
$$
(as $1-1/x\le \log x$ and so $1/(1+x)\le \exp(-x/(1+x))$ for $x>0$). For $t_0=1/C_R$,
$
\Xi_n(C_R)\ge \frac{1}{C_R}(1-\exp(-1/2)) \ge \frac{1}{3\,C_R}.
$
It is well known (\cite{gordon1941values}) that  $\Phi(-x)\ge \frac{1}{\sqrt{2\pi}}\exp\pp{-x^2/2} \frac{1}{x+1/x}$ for $x>0$, and so  
$$ 
\delta
=2\,\Phi\left(-r\, \sqrt{3 \, C_R}\right) 
\ge \frac{2}{\sqrt{2\pi}(1+r\sqrt{3\,C_R})}  \exp\left(-\frac{3\,r^2}{2} \, C_R\right)\eqcolon \tilde \delta,
$$
for $r\sqrt {3 C_R}\ge 1$. We conclude that $\PP(\tilde E)\le 1-\tilde \delta$.
\end{proof}

\begin{corollary}[coupling] \label{minorisation} Under the assumptions of \cref{coupling}, for all $r>0$, there exists $t_0>0$, so that
$\PP(X^n=Y^n) 
\ge \alpha>0$,
 uniformly for $x,y\in B(0,r)$ and $0<\Delta t\le t_0$ with $n\Delta t=t_0$. 
 \end{corollary}
\begin{proof}
Let 
$E=\{X^1,\dots,X^{n-1},Y^1,\dots,Y^{n-1}\in B(0,R)\}$ and $\tilde E=E\cap \{X^n\ne Y^n\}$ where $n\,\Delta t=t_0$. Then, if $C_R$ is sufficiently large, $t_0=1/C_R$ and $\Delta t\in(0,t_0]$, we have $\PP( E)\ge 1-\delta$ and $\PP(\tilde E)\le 1-\tilde\delta$ by \cref{coupling,escape}, with
\begin{align*}
\delta = \rev{2\exp\pp{-\frac{C_R R^2}{2C}}
\exp\pp{ \frac{ C_R \,r^2}{C} +2}},\qquad 
\tilde\delta= \frac{2}{\sqrt{2\pi}(1+r\sqrt{3\,C_R})}  \exp\left(-\frac{3\,r^2}{2} \, C_R\right).
\end{align*}
Notice that $\tilde E^c$ is the union of $\{X^n=Y^n\}$ and $E^c$. Thus,
$\PP(X^n=Y^n) \ge \PP(\tilde E^c)-\PP(E^c)
\ge \tilde\delta- \delta\eqcolon\alpha$. 
Then,
\begin{align*}
\alpha
&= \frac{2}{\sqrt{2\pi}(1+r\sqrt{3\,C_R})}  \exp\left(-\frac{3\,r^2}{2} \, C_R\right)
  -\rev{  2\exp\pp{-\frac{C_R R^2}{2C}}
\exp\pp{ \frac{ C_R \,r^2}{C} +2}}.
  \end{align*}
  \rev{  Neglecting terms in $r$, and using the scaling $R^2 \lesssim C_R \lesssim R^2$ (see Lemma \ref{Lip_f}), we may choose $R$ large enough so $\alpha>0$.}
\end{proof}

We now show the drift and coupling  conditions imply geometric convergence to the invariant measure. 
\begin{proof}[Proof of \cref{NGE}]
From Lemma \ref{lem:2ndmoments_discrete}, 
there exists $\nu>0$ and $C>0$ such that, for $V(x) = \|x\|^2 + 1$ and $V_4(x) = \|x\|^4+1$,
$$P^{n}(V)(x) \leq (1-\nu\,\tstep)^n V(x) + C\,n\,\tstep,\qquad P^{n}(V_4)(x) \leq (1-\nu\,\tstep)^n V_4(x) + C\,n\,\tstep.$$
 The constants $\nu$ and $C$ are independent of $\tstep$ sufficiently small. 
Choose $r>0$ large enough that  $-\frac12\nu + 2C /r^2 <0$. Corresponding to this $r$, choose $t_0$ according to \cref{minorisation}. Then, for $n=\lceil t_0/\deltat\rceil$,
\begin{align}\label{consequence_drift_cond}
P^{n}(V)(x) \leq \gamma \,V(x) + C_1,\qquad P^{n}(V_4)(x) \leq \gamma\, V_4(x) + C_1,
\end{align}
with $\gamma=\exp(-\nu \,t_0)$ and $C_1=C\,t_0$ and $\tstep\le 1/2\nu$ (as $(1-\nu\, \tstep)^n\le \exp(-\nu\,t_0)$). Later, we will exploit the inequality $\gamma+2C_1/r^2<1$ (which holds as $\exp(-x)\le 1-x/2$ for $0<x\le 1$).

We closely follow the proof of \cite{hairer2011harris}: the first part A), where we consider the setting of $V(x) + V(y)\geq r^2$, uses only the drift condition and is identical but repeated here verbatim only for completeness; the second part, for the case $V(x) + V(y)> r^2$, makes a small adaptation of the proof to exploit our coupling probability.

\textbf{A) Using the drift condition.}
    To show the contraction \cref{contraction} when $V(x) + V(y)\geq r^2$, we simply make use of the drift condition and choose $\BETA$ accordingly:
    \begin{align*}
        P^{n}(\phi)(x)-P^{n}(\phi)(y) &\leq 2+\BETA\, P^n(V)(x) + \BETA\, P^n(V)(y)
        \leq 2 + \BETA \,\gamma\,(V(x)+V(y)) + 2\,\BETA\, C_1,
    \end{align*}
    where we make use of $\norm{\phi}_\BETA\leq 1$ for the first inequality and the drift condition for the second inequality. Since $V(x) + V(y)\geq r^2$, by choosing $\gamma_0 = \gamma + 2C_1/r^2 <1$, we have
    $$
    2 + \BETA \gamma(V(x)+V(y)) + 2\BETA C_1 \leq 2 + \BETA\gamma_0(V(x)+V(y)) \leq 2\gamma_1 + \BETA\gamma_1 (V(x)+V(y)),
    $$
    where the last inequality is due to $ 2(1-\gamma_1 )\leq \BETA (\gamma_1-\gamma_0) r^2 \leq   \BETA (\gamma_1-\gamma_0) (V(x)+V(y))$ by choice of $\gamma_1\eqdef\frac{2 + \BETA \gamma_0 r^2}{(2+\BETA  r^2 )}  <1 $.

\textbf{B) Using the coupling  probability.}
    We now consider the setting where $V(x)+V(y)\leq r^2$, so $x,y\in B(0,r)$ and \cref{minorisation} yields $\PP(X^n\ne Y^n) \le 1-\alpha$.   
Given a test function $\phi$ such that $\norm{\phi}_\BETA \leq 1$,
    \begin{align*}
     \EE[(\phi(X^{n}) - \phi(Y^{n})) 1_{\{X^n\neq Y^n\}}]
     &
    \leq  2\,\PP(X^n\neq Y^n) + \BETA\,\EE\bp{(V(X^n) + V(Y^n)) 1_{\{X^n\neq Y^n\}}
    }\\
    &\leq 2 (1-\alpha) +\BETA \sqrt{(1-\alpha)}  \sqrt{\EE[V(X^n)^2 + V(Y^n)^2]}.
    \end{align*}
Using the drift condition on $V_4$, we deduce that
    \begin{align*}
    \EE[\phi(X^{n}) - \phi(Y^{n}) ]&\leq 2(1-\alpha) + \BETA \sqrt{(1-\alpha)} \sqrt{\gamma V(x)^2 +\gamma  V(y)^2 + 2C_1}\\
    &\leq 2(1-\alpha) + \BETA \sqrt{(1-\alpha)\gamma} \pa{V(x) + V(y) + \sqrt{2C_1/\gamma}}.
    \end{align*}
    Choose $\BETA$ such that $2(1-\alpha) +\BETA \sqrt{2C_1 (1-\alpha)} <2(1-\tilde \alpha)$ for some $0<\tilde \alpha<\alpha$; that is,
    $   \BETA  < 2(\alpha - \tilde \alpha)/\sqrt{2C_1 (1-\alpha)}$.
    With $\gamma_2 \eqdef \max((1-\tilde \alpha), \sqrt{(1-\alpha)\gamma})$, we now have
    $$
    \abs{P^{n}(\phi)(x)-P^{n}(\phi)(y) }\leq 2(1-\tilde \alpha) + \BETA \sqrt{(1-\alpha)\gamma}(V(x) + V(y)) \leq \gamma_2 d_\BETA(x,y).
    $$
    
     Putting A) and B) together, we see $d_\BETA(P^n(\mu_1,\cdot),P^n(\mu_2,\cdot))\le \bar\alpha\, d_\BETA(\mu_1,\mu_2)$, and $P^n$ is a contraction  with respect to $\norm{\cdot}_\BETA$ with contraction constant $\bar \alpha \eqdef \min(\gamma_1,\gamma_2)$. This implies geometric convergence to the invariant measure in the equivalent weighted TV norm as stated in \cref{finale}.
\end{proof}

\subsection{Finite-time convergence of discretized scheme}\label{sec:disc_conv_inv}

 For $t_k=k\,\tstep$, let  $e_1(t_k) \eqdef u_{t_k} - u^k$ and $e_2(t_k) \eqdef v_{t_k} - v^k$, and $e(t_k) \eqdef (e_1(t_k), e_2(t_k)) $. Our main convergence result is as follows.

\begin{theorem}\label{thm:error_discrete}
    Suppose that \cref{reg3} holds. Then, for any $R>0$,
      $$
\EE\bp{ \sup_{0\leq t_k\leq T} \norm{e(t_k)} }\leq  c_R\,d \,\exp\pa{2\,T\, C_R}  \deltat^{\frac12-\epsilon} +  \rev{ C \,\sqrt{T}\, e^{-C\,(R^2/T)}},
$$
where $c_R$ and $C_R$ depends on $R$.  We can choose $R$ and $\tstep$ to make the right-hand side arbitrarily small, and hence $\EE\bp{ \sup_{0\leq t_k\leq T} \norm{e(t_k)} }\to 0$ as $\tstep\to 0$.
\end{theorem}

The rest of this subsection is dedicated to proving \cref{thm:error_discrete}. 
Before proceeding with the proof, we first point out the two main technical challenges that we address.

\textit{Limited inverse moments.} As mentioned, our implicit-explicit scheme with $G\equiv 0$ is  the same as the scheme proposed in \cite{Dereich2011-tp}, who studied the discretization scheme for the Lamperti-transformed version of the CIR process \eqref{eq:cir1} under the parameter choice $2\kappa\mu>\theta^2$. Their main result is the strong error bound $\EE\biggl[\sup_{0\leq t_k \leq T} \norm{ u_{t_k} - u^k}\biggr] = \Oo(\deltat^{\frac12}).$ Their analysis required inverse-moment bounds of the form $\EE[u_t^{-2}]<\infty$: this is \textit{false} when $2\kappa\mu = \theta^2$, which is precisely the case for our $u$ process when $G\equiv 0$. 

    To overcome this, we derive inverse-moment bounds of order up to but excluding 2 (we will make use of \cref{cor} and \cref{timeregbnd} in our proof). As a result, in the case where $G\equiv 0$, our proof shows that the discretized CIR process with parameters $2\kappa\mu =\theta^2$  satisfies for any $\epsilon>0$, the strong error bound $\EE\biggl[\sup_{0\leq t_k \leq T} \norm{ u_{t_k} - u^k}\biggr] = \Oo(\deltat^{\frac12-\epsilon})$. We remark that similar ideas and error bounds were applied in \cite{hutzenthaler2014strong}.

\textit{Lack of global Lipschitzness.} The drift in \cref{discrete_scheme} comprises 
the explicit part $f$ and implicit part $S_\lambda(u) \eqdef \frac{1}{\beta\, u} - \lambda u$. The second difficulty is that  neither $S_\lambda$ nor $f$ are globally Lipschitz.  In particular, $f$ is at best only locally Lipschitz. To handle this, we will follow the analysis of \cite{higham2002strong} to stop our process whenever $\norm{(u,v)}\geq R$ and upper bound $\PP(\norm{(u,v)_t}\geq R)$. Intuitively, this can be done since our stationary distribution has exponentially decaying tails.

The proof of \cref{thm:error_discrete}  is a consequence of \cref{error_under_bounded_ball} below, which shows convergence for a stopped process confined to a ball of radius $R$. While we manage to show $\tstep^{1/2-\epsilon}$ rate for a stopped process (restricted to the ball of radius $R$), our main convergence result, \cref{thm:error_discrete}, does not provide a rate because our error bound has an exponential dependence on $C_R$ and hence $R$.

Define the following stopping times
\begin{align}\label{st}
    k_R \eqdef \inf\enscond{k\geq 0}{\snorm{( u, v)^k} \geq R}, \quad \gamma_R \eqdef \inf\enscond{t\geq 0}{\norm{(u,v)_t} \geq R}
\end{align}
and $\theta_R \eqdef k_R \deltat \wedge \gamma_R$.
We have estimates on the tail behaviour of $k_R$ in \cref{escape};  similar estimates holds for $\gamma_R$.
\rev{
\begin{lemma}\label{cts-escape} For some $c_1,c_2,c_3>0$ independent of $T$,$R$ and $r$, the following holds for $x\in B(0,r)$,
    $$
    \PP\pp{\gamma_R<T}\le \exp(-c_1\,R^2)\,\exp(c_2 (r^2+c_3\,T)).
    $$
\end{lemma}
}
\begin{proof}
    A calculation with It\^o's formula for $V(u,v)=1+\|(u,v)\|^2$ shows that, for some $C>0$
    $
    dV(x_t)\le - 2\,\lambda\, V(x_t)\,dt+C\,dt+ d M_t,$
    where $M_t=2\sqrt{2/\beta}\int_0^ t\langle{x^t, dW_t}\rangle$. This is a martingale with  quadratic variation
    $
    \langle M\rangle_t
    =\frac{8}{\beta}\int_0^t \norm{x_s}^2 \,ds
    $.
    For $\alpha=\lambda\,\beta/4$,
    \begin{align*}
    V(x_t)&\le V(x_0)-2\,\lambda \int_0^t V(x_s)\,ds+\frac{1}{2}\alpha\frac{8}{\beta} \int_0^t \norm{x_s}^2\,ds+C\, t+M_t-\frac{1}{2}\alpha\langle M\rangle_t\\
    &\le V(x_0) +C\,t +M_t -\frac 12\alpha \langle M\rangle_t.
    \end{align*}
    Then, by Doob's martingale inequality,
     $$
    \PP\pp{\sup_{0\le t\le T}\pp{V(x_t)-V(x_0)- C\,t}  \ge  R} \le
    \PP\pp{\sup_{0\le t\le T}\pp{M_t-\frac 12 \alpha\langle M\rangle_t}  \ge  R} \le
    \exp(-R\,\alpha).
    $$
    As $\gamma_R$ is the time of first escape from $B(0,R)$,  $V(x_{\gamma_R})=R^2+1$ and 
   \begin{align*}
    \PP\pp{\gamma_R<T}&\le
    \PP\pp{\sup_{0\le t\le T}\pp{V(x_t)-V(x_0)- C\,t}  \ge  R^2-\norm{x_0}^2-c\,T}\le \exp(-(R^2-r^2-cT)\lambda\,\beta/4)\\
    &\le \exp(-R^2\,\lambda\,\beta/4)\,\exp((r^2+C\,T)\lambda\,\beta/4).\qedhere
    \end{align*}
\end{proof}
Up to time $\theta_R$, we have the following error.
\begin{proposition}\label{error_under_bounded_ball}
  Let the assumptions of \cref{thm:error_discrete} hold and $\deltat \leq \frac{1}{2C_R}$.  
    For all $\epsilon>0$, 
    \begin{equation}
        \EE\bp{\sup_{0\leq t_k\leq T} \norm{e(t_k\wedge \theta_R)}} \leq  c  \,\exp(2\,T\, C_R)\, \deltat^{\frac12-\epsilon},
    \end{equation}
    where $c$ that depends on $\epsilon$, $C_R$, $d$, and $\lambda$.
\end{proposition}
\begin{proof}
In the following, assume that $t_k \leq \theta_R$. 

We first subtract the discretized equations with  the continuous counterparts,
\begin{align*}
&u_{t_{k+1} } = u_{t_k} +  \int_{t_k}^{t_{k+1}}  \left( -\lambda u_t + v_t\odot \nabla G(u_t\odot v_t) + \frac{1}{u_t}\right) dt + \sqrt{2/\beta}\int_{t_k}^{t_{k+1}}  dW^1_t,\\
&v_{t_{k+1} } = v_{t_k} +  \int_{t_k}^{t_{k+1}}  \left( -\lambda v_t + u_t\odot G(u_t \odot v_t) \right) dt + \sqrt{2/\beta}\int_{t_k}^{t_{k+1}} dW^2_t,
\end{align*}
to obtain
\begin{align}
    \begin{split}    
    e_1(t_{k+1}) &= e_1(t_{k}) - \deltat \,\pa{f_1((u,v)_{t_k}) - f_1((u,v)^k)  }+ \deltat \,\pa{S_\lambda(u_{t_{k+1}}) - S_\lambda(u^{k+1})} + r_{1,k},
    \end{split}    \label{eq:err1}\\
    \begin{split}
    e_2(t_{k+1}) &= e_2(t_{k}) - \lambda\,\deltat\, e_2(t_{k+1})  - \deltat \,\pa{f_2((u,v)_{t_k}) - f_2((u,v)^k)  }  + r_{2,k},
   \end{split}\label{eq:err2}
\end{align}
where
\begin{align*}
r_{1,k} &\eqdef  \int_{t_k}^{t_{k+1}}  \pa{ f_1((u,v)_{t_k}) - f_1((u,v)_t) } dt + \int_{t_k}^{t_{k+1}} \left( S_\lambda(u_t) - S_\lambda(u_{t_{k+1}})\right) \,dt,\\
r_{2,k} &\eqdef \int_{t_k}^{t_{k+1}}  \pa{ f_2((u,v)_{t_k}) - f_2((u,v)_t) } \,dt - \lambda \int_{t_k}^{t_{k+1}} \left( v_t - v_{t_{k+1}} \right)\, dt.
\end{align*}
Taking inner product of \eqref{eq:err1} with $e_1(t_{k+1})$ and \eqref{eq:err2} with $e_2(t_{k+1})$, we have 
\begin{align*}
    \pa{\frac12 + \lambda \,\deltat - \frac{ C_R\,\tstep}{2}} \norm{e_1(t_{k+1})}^2 & \leq \frac12\norm{e_1(t_{k})}^2+ \frac{ C_R\,\tstep }{2} \norm{(e_1(t_{k}),e_2(t_{k}))}^2   + \norm{r_{1,k}}\,\norm{e_1(t_{k+1})}, \\
   \left(\frac12+\lambda \,\deltat - \frac{C_R\,\tstep}{2} \right) \norm{e_2(t_{k+1})}^2 & \leq \frac12 \norm{e_2(t_{k})}^2+   \frac{C_R\,\tstep }{2} \norm{(e_1(t_{k}),e_2(t_{k}))}^2   + \norm{r_{2,k}}\,\norm{e_2(t_{k+1})},
\end{align*}
where we used the Cauchy--Schwarz inequality, that $f_i$ is $C_R$-Lipschitz, and that  $$\bigl(S_\lambda(u_{t_{k+1}}) - S_\lambda(u^{k+1})\bigr)^\top e_1(t_{k+1}) \leq -\lambda \norm{e_1(t_{k+1})}^2$$
 due to monotonicity arguments. Denote $E_k = \norm{(e_1(t_{k}),e_2(t_k)) }$ and  $R_k= \norm{(r_{1,k},r_{2,k})}$ and sum the two equations to find 
\begin{align*}
    \pa{1 + 2\,\lambda \,\deltat -  C_R\,\tstep} E_{k+1}^2 \leq \pa{1+ 2\,C_R \,\deltat } E_k^2 +   2\,E_{k+1} R_k .
\end{align*}
Define $c_1 = \pa{1+ 2\,C_R\, \deltat }/ \pa{1 + 2\lambda\, \deltat - \,C_R\,\tstep}$ and $c_2 = 2/ \pa{1 + 2\,\lambda\, \deltat -  C_R\,\tstep}$. Then, for $t_{k+1}\le \theta_R\wedge M$,
\begin{align*}
    E_{k+1}^2 &\leq c_1 E_k^2 + c_2 E_{k+1} R_k \leq c_1^{k+1} E_{0}^2 + c_2 \sum_{j=0}^{k} c_1^j  E_{k+1-j} R_{k-j}\leq c_2 \sup_{k\leq M\wedge \theta_R} E_k \sum_{j=0}^{k} c_1^j   R_{k-j}.
 \end{align*}
 We therefore have
 $
 \EE\biggl[\sup_{k\leq M \wedge\theta_R} E_k\biggr]\leq c_2 \sum_{j=0}^{M-1} c_1^j \,\EE[R_{k-j}].
 $
 Note that for $\deltat \leq 1/(2C_R)$, we have $c_1\leq (1+2\,\deltat\, C_R)^2$: this is a consequence of the simple inequality $(1-a)^{-1}\leq 1 + 2a$ applied to $a = \deltat\, C_R \leq 1/2$.
Also using the geometric sum formula and the inequality $e^z \geq 1 + z$ for $z\geq 0$, we get
  $$\sum_{j=0}^{M-1} c_1^j \leq \sum_{j=0}^{M-1}  (1+ 2\,\tstep \,C_R)^{2j} \le\frac{\exp(2\, C_R\,M\,\tstep)}{4\, C_R\,\tstep +  4\,C_R^2\,\tstep^2}.$$ 
 On the other hand, $ \EE[R_k] \leq \EE[\norm{r_{1,k}}] + \EE[\norm{r_{2,k}}]$.
 Observe that by \cref{timeregbnd} and the definition $S_\lambda$ in \eqref{eq:f},
 \begin{align*}
   \EE \norm{r_{1,k}}
   &\leq \int_{t_k}^{t_{k+1}} \EE\norm{f_1((u,v)_{t_k}) - f_1((u,v)_t)} + \lambda \EE \norm{u_t - u_{t_{k+1}}}\, dt
    +\frac{1}{\beta} \sum_{i=1}^d  \int_{t_k}^{t_{k+1}}
    \EE\bp{\frac{\abs{u_{i,t}- u_{i,t_{k+1}}}}{ u_{i,t} u_{i,t_{k+1}}}}\, dt\\
    &\leq c\pa{ (C_R+\lambda)\,\deltat^{3/2} + d\,\deltat^{\frac32 -\epsilon}}.
\end{align*} 
 Similarly,
$
\EE[\norm{r_{2,k}}] \leq c\, (C_R+\lambda) \,\deltat^{3/2}.
$
It follows that
$
\EE\biggl[{\sup_{0\leq t_k\leq T\wedge \theta_R} E_k}\biggr]
\leq c_3\,\exp( 2\,C_R\,T) \,\deltat^{\frac12-\epsilon},
$ for a constant $c_3$ that depends on $\lambda$, $C_R$, and $d$, as required.
\end{proof}

We are now ready to prove the main result.
\begin{proof}[Proof of \cref{thm:error_discrete}]
Define $\tau_R\eqdef k_R\deltat$ (see \cref{st}). 
 First note that $ \EE\bp{ \sup_{0\leq t_k\leq T} \norm{e(t_k)}}\leq A+B$ where $A\eqdef  \EE\bp{\sup_{0\leq t_k\leq T} \norm{e(t_k)} 1_{\ens{\tau_R>T, \gamma_R>T}}}$ and $B\eqdef  \EE\bp{ \sup_{0\leq t_k\leq T} \norm{e(t_k)} 1_{\ens{\tau_R< T \text{ or } \gamma_R<T}}}$. By \cref{error_under_bounded_ball}, $A\leq c \,\exp\pa{2\,T\, C_R} \,d \,\deltat^{\frac12-\epsilon} $. Using the Cauchy--Schwarz inequality, we bound $B$ as $ B\leq  \sqrt{\EE\bp{ \sup_{0\leq t_k\leq T} \norm{e(t_k)}^2}} \sqrt{\PP(\gamma_R<T) + \PP(\tau_R<T)}$.
          From \eqref{eq:prob_escap_pre} with \rev{$\alpha=1/(CT)$},
    $$
    \EE\bp{\sup_{0\le t_k\le T}\snorm{(u,v)^k}^2} \leq \snorm{(u,v)^0}^2 + \rev{2}C\,T + \int_0^\infty e^{-\alpha \,r} \,dr,
    $$
    which is uniformly bounded with respect to $\tstep$.  Finally, applying \cref{escape,cts-escape}, \rev{$B\leq C \sqrt{T} e^{-C(R^2/T) }$ for some constant $C>0$}
    which completes the proof. 
    
    \end{proof}

\subsection{Invariant measure}
To  conclude this section, we combine \cref{NGE}  and \cref{main} to show that the stationary distribution $\pi_\tstep$ of $(u,v)^k$ converges to $\pi$. Here $W_1$ denotes the usual Wasserstein 1-norm.

\begin{theorem}[convergence to invariant measure]\label{thm_convergence_to_invariant3}
  Let \cref{reg3} hold. Then, $W_1(\pi_{\tstep}, \pi) \to 0$ as $\tstep\to 0$.
  \end{theorem}
  \begin{proof} 

  By geometric ergodicity of the numerical approximation in \cref{NGE} with $t_k=T$, we have $|\int \phi(x) \,P^k(\pi,dx)-\int\phi(x) \pi_{\tstep}(dx)|\le\bar \rho^T(\norm{\pi}_\BETA+\norm{\pi_\tstep}_\BETA)$ for any $1$-Lipschitz function $\phi$. As $\pi(V)$ and $\pi_{\tstep}(V)$ are bounded uniformly in $\deltat$,  we may choose $c_3$ independent of $\tstep$ such that $\norm{\pi}_\BETA+\norm{\pi_\tstep}_\BETA\le c_3$ and
\begin{equation}
    W_1( P^k(\pi,\cdot),\pi_{\tstep})\le c_3 \bar\rho^T.\label{rain}
\end{equation}

 Denote by $ P^t$ the transition density of the solution $x^t$ of \cref{main}. The initial distribution $\rho_0=\pi$ satisfies Assumption~\ref{reg3}(ii) and \cref{thm:error_discrete} gives
  \begin{equation}
\sup_{0\le t_k\le T}  W_1( P^k(\pi,\cdot), P^{t_k}(\pi,\cdot))
  \to 0,\qquad \text{as $\tstep\to 0$.}\label{thunder}
  \end{equation}
  Divide the error into finite-time numerical error and distance to the long-time average:
\begin{equation*}
W_1(\pi,\pi_{\tstep})
=  W_1( P^k(\pi,\cdot), P^{t_k}(\pi,\cdot))+W_1( P^k(\pi,\cdot),\pi_{\tstep}).
\end{equation*}
We now choose $T=\abs{\log(\delta/2\,c_3)}/\abs{\log\bar\rho}$, in \cref{rain} with $t_k=T$ to ensure 
  $W_1( P^k(\pi,\cdot),\pi_{\tstep})\le  \frac12\delta$; and the timestep $\tstep$ in \cref{thunder} so that  $W_1( P^k(\pi,\cdot), P^{t_k}(\pi,\cdot))
  \le  \frac 12\delta$. We see now that $
W_1(\pi,\pi_{\tstep})\le \delta$.
  \end{proof}

\begin{figure}
    \centering
    \includegraphics[width=0.9\linewidth]{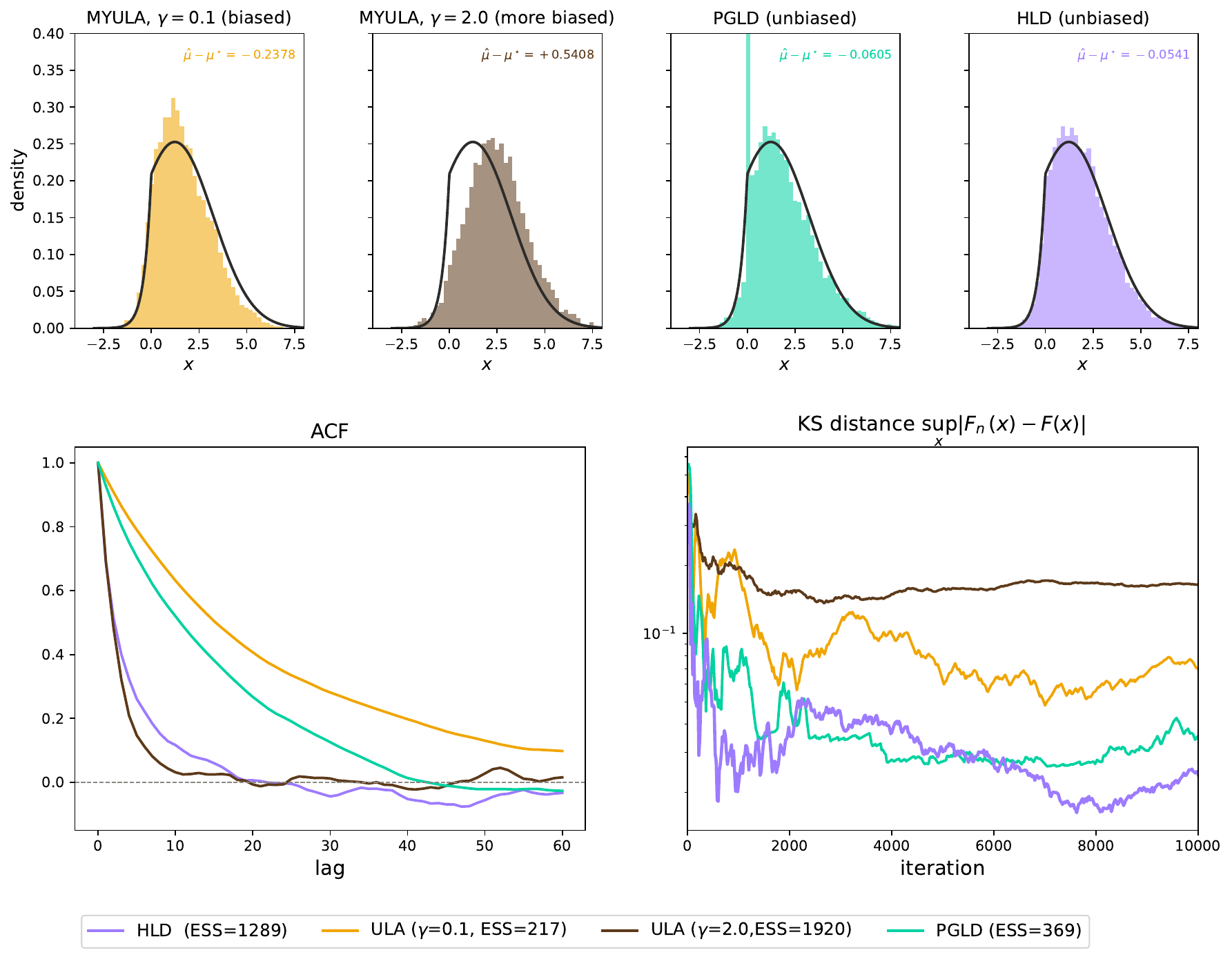}
    \caption{Top row: histograms of MYULA  \citep{pereyra2016proximal} with $\gamma=0.1$ and $\gamma = 2$,  proximal gradient Langevin dynamics (PGLD) \citep{durmus2019analysis} and Hadamard--Langevin dynamics (HLD) after $10^3$ iterations against the true density (in black). For MYULA, the stepsize is chosen to be $\gamma/(1+4L\gamma)$, while SPGLD and HLD use stepsize $0.1/L$ where $L$ is the Lipschitz constant of $G$. Bottom row: plot of the Autocorrelation Function (ACF) against lag and  the maximum vertical distance between the empirical and true cumulative distribution function. The regularization $\gamma$ for MYULA ensures a smoother invariant distributions and has faster mixing but is biased (as demonstrated by the higher effective sample size (ESS) score but larger KS distance). While PGLD has no smoothing bias, under finite stepsize, the soft thresholding operator means that  the stationary distribution of PGLD has a
point mass at zero mixed with a continuous density, hence the observed spike at 0 in the histogram. In contrast, Hadamard has no bias, and appears to have faster mixing compared to PGLD. A theoretical investigation of this phenomenon would be an interesting venue for future work. }
    \label{fig:mixing}
\end{figure}

\section{Conclusion and outlook}
We proposed a new approach for sampling distributions with the Laplace prior via a Hadamard parameterization. Our resulting Hadamard--Langevin Dynamics can be seen as an implicit form of Riemannian Langevin. The analysis in this work covered well-posedness of the continuous-time dynamics, strong convergence of the discretized dynamics, as well as geometric ergodicity results in both continuous and discrete time. These results provide a theoretical foundation for developing sampling schemes based on overparameterized Langevin dynamics. In the following, we detail the main remaining missing piece in our analysis of the Hadamard parametrization, and the challenges in extending this to other regularizers. 

\textbf{Theoretical limitations.}
On a theoretical level, for our numerical scheme, we established a strong convergence result (i.e., convergence as stepsize diminishes over a fixed time interval), and convergence of the  invariant distribution $\pi_\tstep$ for the discretized process. However, we are not able to establish quantitative rates of convergence. As mentioned in \cite[Section 3]{Lelievre2016-hx}, such convergence rates are typically obtained by understanding the Taylor expansion of the semigroup operator $P_{\deltat}\phi = \EE[\phi(x_{t_{k+1}})\mid x_{t_k} = \cdot]$. The difficulty of a theoretical proof is that these Taylor expansions would include singular $1/u$ terms and this lack of differentiability prevent the direct use of results such as \cite[Theorem 3.3]{Lelievre2016-hx}.  One important direction of future work is to quantify the bias of $\pi_{\deltat}$ and to provide quantitative strong convergence rates.  

\textit{One-dimensional experiment.}
In the absence of a  theoretical characterization, we provide  some numerical evidence on a simple one-dimensional problem in Figure \ref{fig:rate}  below that this bias is $\mathcal{O}(\deltat)$.

We consider the case where $d=1$, and choose $A=1$, $y=3$, and $\lambda = 2.7$.  We let $G(x) = \|x-y\|^2/2$.
We carry out two experiments with $\phi(x) = x^2$.
To estimate the impact of $\Delta t$ for the stationary distribution, for different stepsizes, we run our Hadamard--Langevin (HLD) algorithm for $T=800$ before recording $10^6$ samples. Figure \ref{fig:rate}(a) shows the error between the mean of  $\phi(u\odot v)$  evaluated via these samples compared with the integral of the true density $\rho$ (i.e., \eqref{target}) against $\phi$ (we compute this numerically  using scipy.integrate). This procedure is also repeated for MYULA, where the Moreau-envelope regularization choice  are taken to be $\gamma=1/(KL)$ for some $K\geq 1$ and $\deltat = \gamma/(5(\gamma L +1))$, where $L= \norm{A}^2$ is the Lipschitz constant of the gradient of $G$. This is consistent with the parameter choices recommended in \cite{durmus2018efficient}. We observe $\Oo(\deltat)$ errors for MYULA, and just slightly slower than $\Oo(\deltat)$ for HLD, although the error is typically lower for moderate stepsizes.

 To assess the convergence to the stationary distribution with fixed stepsize, fix $\Delta t=5\times 10^{-4}$, and generate $5\times 10^4$  paths of HLD and record the means $\phi(u\odot v)$ over these paths  for $10^4$ iterations. Figure \ref{fig:rate}(b) shows the error between the mean of $\phi(u\odot v)$ over the first $3000$ iterations against the  mean of  $\phi(u\odot v)$ over the last $1000$ iterations (which we expect to estimate the integral of $\phi$ against the  stationary distribution). The same procedure is repeated for Proximal Langevin. In both cases, we observe exponential decay of the error, which corroborates the exponential ergodicity result of \cref{NGE}.

\begin{figure}
    \centering
    \begin{tabular}{cc}
         (a)&(b)\\\includegraphics[width=0.35\linewidth]{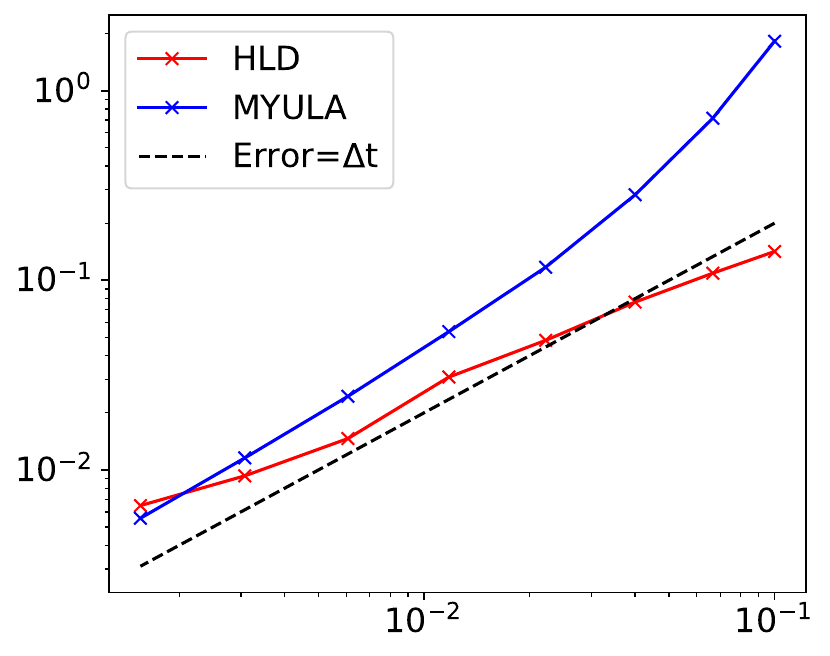}  &\includegraphics[width=0.35\linewidth]{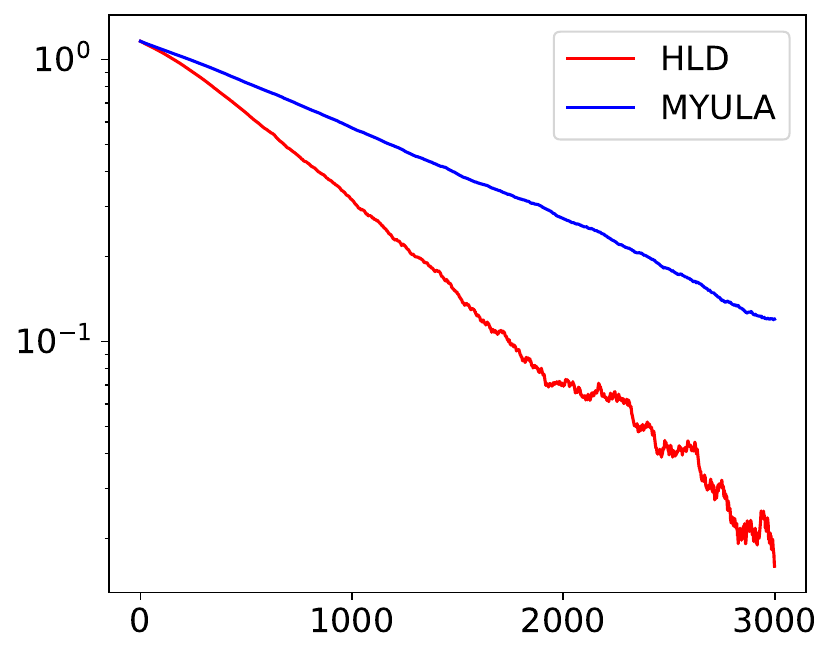}
         \\[-.5em]
         {\small \hspace{2em} $\deltat$}& 
         {\small \hspace{2em} Iteration}
    \end{tabular}
    \caption{Estimation of the convergence rate for a 1D problem.  (a) Error of the `stationary distribution' for different stepsizes. (b) Error decay with increasing iterations  for fixed stepsize.}
    \label{fig:rate}
\end{figure}

The code for reproducing this numerical experiment and also further higher-dimensional implementations of our method can be found online\footnote{ \url{https://hadamardlangevin.readthedocs.io/}} and \cite[Supplementary D]{supple_poon}.

\textbf{Other sparse priors.}
Although this  study is specific to the Laplace prior, one central idea is the development of Langevin dynamics based on Gaussian mixture representation of this prior. We first remark that the extension to the group-$\ell_1$ prior is trivial.
The group-$\ell_1$ prior is $\norm{x}_{1,2}\eqdef \sum_{j=1}^K \norm{x_{\mathbf b_j}}$, where $\cup_{j=1}^K \mathbf b_j$ is a partition of $\ens{1,\ldots, d}$. In this case, one can consider for $(u,v)\in\mathcal{X} = \real_+^K \times \RR^d$ 
$$
    \pi(u,v) \propto \prod_{i=1}^K u_i \exp\pa{-\beta\pa{\frac{\lambda}{2}\norm{(u,v)}^2 + G(u\odot v)  }}, 
$$
where $u\odot v \eqdef (u_i \,v_{\mathbf b_i})_{i=1}^K$. Samples of $\pi(x) \propto \exp\pa{-\beta \norm{x}_{1,2} - \beta G(x) }$ can be retrieved by sampling from $\pi$ and forming the product $u\odot v$ as before.  For this setting, $u\in\RR^K_+$ while $v\in\RR^d$ with $K\leq d$, and the only difference here is the different definition of $\odot$. Likewise,
the Langevin dynamics are identical to the $\ell_1$ setting, except $\odot$ and the different dimensions of $u$ and $v$.

A natural question is how to handle priors corresponding to  $\ell_q$ semi-norms with $q\in (0,1)$. These priors also have a Gaussian mixture representation \citep{west1987scale} $$
    e^{-|z|^q} \propto \int_0^\infty \frac{1}{\sqrt{2\pi \eta}}\, e^{-z^2/(2\eta)} \,\frac{1}{\eta^{3/2}} \,g_{q/2}\pa{\frac{1}{2\eta}}\, d\eta,$$
    where $g_{q/2}$ is the density of a positive stable random variable with index $q/2$. Unfortunately, this has no closed-form expression in general so it is unclear how to perform Langevin sampling for such priors. Nonetheless, given the developments of our present work, it is perhaps conceivable that one could  construct new sparsity-inducing priors and efficient Langevin sampling schemes via Gaussian-mixture representations, and we believe that this is an exciting direction for future work. In particular, recent works such as \citep{calvetti2024computationally} derived generalizations of the hierarchical samplers for the Bayesian lasso to  new sparsity inducing Gaussian mixture models. Our article  can be seen as a first attempt at deriving Langevin dynamics for such priors.

\bibliographystyle{plainnat}
\bibliography{references}

@article{supple_poon,
  title={Supplement to Hadamard--Langevin dynamics for sampling the l1-prior},
  author={Cheltsov, Ivan and Cornalba,Federico and Poon, Clarice and Shardlow, Tony},
  journal={},
  year={2026},
  doi={},
  publisher={},
}

@article{gordon1941values,
  author = {Gordon, R. D.},
  title = {Values of {Mills'} Ratio of Area to Bounding Ordinate and of the Normal Probability Integral for Large Values of the Argument},
  journal = {The Annals of Mathematical Statistics},
  year = {1941},
  volume = {12},
  number = {3},
  pages = {364--366},
  doi = {10.1214/aoms/1177731721},
  url = {http://www.jstor.org/stable/2235868}
}

@article{Lamba2007AdaptiveEulerMaruyama,
  author    = {H. Lamba and J. C. Mattingly and A. M. Stuart},
  title     = {An adaptive {Euler}--{Maruyama scheme} for {SDE}s: convergence and stability},
  journal   = {IMA Journal of Numerical Analysis},
  volume    = {27},
  number    = {3},
  pages     = {479--506},
  year      = {2007},
  month     = {July},
  doi       = {10.1093/imanum/drl032}
}

@article{chizat2022convergence,
  title={Convergence rates of gradient methods for convex optimization in the space of measures},
  author={Chizat, L{\'e}na{\"\i}c},
  journal={Open Journal of Mathematical Optimization},
  volume={3},
  pages={1--19},
doi={10.5802/ojmo.20},
  year={2022}
}

@misc{debortoli2020convergence,
      title={Convergence of diffusions and their discretizations: from continuous to discrete processes and back}, 
      author={Valentin {De Bortoli} and Alain Durmus},
      year={2020},
      eprint={1904.09808},
      archivePrefix={arXiv},
      primaryClass={math.PR},
doi={10.48550/arXiv.1904.09808}
}

@article{fruehwirth2024ergodicity,
  title={Ergodicity of {L}angevin dynamics and its discretizations for non-smooth potentials},
  author={Fruehwirth, Lorenz and Habring, Andreas},
  journal={arXiv},
  year={2024},
doi="10.48550/arXiv.2411.12051"
}

@article{xifara2014langevin,
  title={Langevin diffusions and the Metropolis-adjusted Langevin algorithm},
  author={Xifara, Tatiana and Sherlock, Chris and Livingstone, Samuel and Byrne, Simon and Girolami, Mark},
  journal={Statistics \& Probability Letters},
  volume={91},
  pages={14--19},
doi={10.1016/j.spl.2014.04.002},
  year={2014},
  publisher={Elsevier}
}

@article{durmus2019analysis,
  author = {Alain Durmus and Szymon Majewski and Błażej Miasojedow},
  title = {Analysis of {Langevin Monte Carlo} via convex optimization},
  journal = {Journal of Machine Learning Research},
  volume = {20},
  number = {73},
  pages = {1--46},
  year = {2019},
doi={10.48550/arXiv.1802.09188}
}

@article{calvetti2024computationally,
  title={Computationally efficient sampling methods for sparsity promoting hierarchical {B}ayesian models},
  author={Calvetti, Daniela and Somersalo, Erkki},
  journal={SIAM/ASA Journal on Uncertainty Quantification},
  volume={12},
  number={2},
  pages={524--548},
  year={2024},
  publisher={SIAM},
doi={10.1137/23m1564043}
}

@article{feller1951two,
  title={Two singular diffusion problems},
  author={Feller, William},
  journal={Annals of Mathematics},
  volume={54},
  number={1},
  pages={173--182},
  year={1951},
  publisher={JSTOR},
doi={10.2307/1969318}
}

@book{carlo2001monte,
  title={Monte Carlo methods in financial engineering},
  author={Glasserman, Paul},
  volume={53},
  year={2004},
  publisher={Springer}
}

@inproceedings{daubechies2008iteratively,
  title={Iteratively re-weighted least squares minimization: Proof of faster than linear rate for sparse recovery},
  author={Daubechies, Ingrid and DeVore, Ronald and Fornasier, Massimo and Gunturk, Sinan},
  booktitle={2008 42nd Annual Conference on Information Sciences and Systems},
  pages={26--29},
  year={2008},
  organization={IEEE},
  doi       = "10.1109/ciss.2008.4558489",
  isbn      =  9781424422463
}

@article{hoff2017lasso,
  title={Lasso, fractional norm and structured sparse estimation using a {H}adamard product parametrization},
  author={Hoff, Peter D},
  journal={Computational Statistics \& Data Analysis},
  volume={115},
  pages={186--198},
  year={2017},
  publisher={Elsevier},
  doi       = "10.1016/j.csda.2017.06.007",
  issn      = "0167-9473,1872-7352",
  language  = "en"
}

@article{black1996unification,
  title={On the unification of line processes, outlier rejection, and robust statistics with applications in early vision},
  author={Black, Michael J and Rangarajan, Anand},
    journal   = "Int. J. Comput. Vis.",
  volume={19},
  number={1},
  pages={57--91},
  year={1996},
  publisher={Springer},
  doi       = "10.1007/bf00131148",
  issn      = "0920-5691,1573-1405",
  language  = "en"
}

@article{arviz_2019,
  doi = {10.21105/joss.01143},
  url = {https://doi.org/10.21105/joss.01143},
  year = {2019},
  publisher = {The Open Journal},
  volume = {4},
  number = {33},
  pages = {1143},
  author = {Ravin Kumar and Colin Carroll and Ari Hartikainen and Osvaldo Martin},
  title = {ArviZ a unified library for exploratory analysis of {B}ayesian models in {P}ython},
  journal = {Journal of Open Source Software}
}

@article{bach2012optimization,
  title={Optimization with sparsity-inducing penalties},
  author={Bach, Francis and Jenatton, Rodolphe and Mairal, Julien and Obozinski, Guillaume and others},
  journal={Foundations and Trends{\textregistered} in Machine Learning},
  volume={4},
  number={1},
  pages={1--106},
  year={2012},
  publisher={Now Publishers, Inc.},
  doi       = "10.1561/2200000015",
  issn      = "1935-8237,1935-8245",
  language  = "en"
}

@ARTICLE{Poon2023-av,
  title     = "{Smooth over-parameterized solvers for non-smooth structured
               optimization}",
  author    = "Poon, Clarice and Peyr{\'e}, Gabriel",
  journal   = "Math. Program.",
  publisher = "Springer Science and Business Media LLC",
  volume    =  201,
  number    = "1-2",
  pages     = "897--952",
  month     =  sep,
  year      =  2023,
  copyright = "https://www.springernature.com/gp/researchers/text-and-data-mining",
  language  = "en",
  issn      = "0025-5610, 1436-4646",
  doi       = "10.1007/s10107-022-01923-3"
}

@article{poon2021smooth,
  title={Smooth bilevel programming for sparse regularization},
  author={Poon, Clarice and Peyr{\'e}, Gabriel},
  journal={Advances in Neural Information Processing Systems},
  volume={34},
  pages={1543--1555},
  year={2021},
url={https://arxiv.org/abs/2106.01429},
doi={10.48550/arXiv.2106.01429}
}

@article{chou2023more,
  title={More is less: inducing sparsity via overparameterization},
  author={Chou, Hung-Hsu and Maly, Johannes and Rauhut, Holger},
  journal={Information and Inference},
  volume={12},
  number={3},
  pages={1437--1460},
  year={2023},
  publisher={Oxford University Press},
  doi       = "10.1093/imaiai/iaad012",
  issn      = "2049-8764,2049-8772",
  language  = "en"
}

@article{vaskevicius2019implicit,
  title={Implicit regularization for optimal sparse recovery},
  author={Vaskevicius, Tomas and Kanade, Varun and Rebeschini, Patrick},
  journal={Advances in Neural Information Processing Systems},
  volume={32},
  year={2019},
  archivePrefix = "arXiv",
  primaryClass  = "stat.ML",
  eprint        = "1909.05122",
doi={10.48550/arXiv.1909.05122}
}

@article{maslowski2000probabilistic,

  title     = "{Probabilistic approach to the strong {F}eller property}",
  author    = "Maslowski, Bohdan and Seidler, Jan",
  journal   = "Probab. Theory Relat. Fields",
  publisher = "Springer Science and Business Media LLC",
  volume    =  118,
  number    =  2,
  pages     = "187--210",
  month     =  oct,
  year      =  2000,
  doi       = "10.1007/s440-000-8014-0",
  issn      = "0178-8051,1432-2064",
  language  = "en"
}

@ARTICLE{Roberts1996-je,
  title     = "{Exponential convergence of {L}angevin distributions and their
               discrete approximations}",
  author    = "Roberts, Gareth O and Tweedie, Richard L",
  journal   = "Bernoulli (Andover.)",
  publisher = "JSTOR",
  volume    =  2,
  number    =  4,
  pages     =  341,
  month     =  dec,
  year      =  1996,
  doi       = "10.2307/3318418",
  issn      = "1350-7265,1573-9759"
}

@article{hutzenthaler2014strong,

  title         = "{Strong convergence rates and temporal regularity for
                   {C}ox--{I}ngersoll--{R}oss processes and {B}essel processes with
                   accessible boundaries}",
  author        = "Hutzenthaler, Martin and Jentzen, Arnulf and Noll, Marco",
  journal       = "arXiv [math.NA]",
  month         =  mar,
  year          =  2014,
  archivePrefix = "arXiv",
  primaryClass  = "math.NA",
  eprint        = "1403.6385",
doi={10.48550/arXiv.1403.6385}
}

@incollection{hairer2011harris,
  title={Yet another look at {Harris}' ergodic theorem for {Markov} chains},
  author={Hairer, Martin and Mattingly, Jonathan C.},
  booktitle={Seminar on Stochastic Analysis, Random Fields and Applications VI},
  series={Progress in Probability},
  volume={63},
  pages={109--117},
  year={2011},
  publisher={Birkh{\"a}user},
  address={Basel},
  isbn={978-3-0348-0020-4},
  doi={10.1007/978-3-0348-0021-1},

}

@article{mattingly2002ergodicity,

  title     = "{Ergodicity for {SDE}s and approximations: locally {L}ipschitz vector
               fields and degenerate noise}",
  author    = "Mattingly, J C and Stuart, A M and Higham, D J",
  journal   = "Stoch. Process. Their Appl.",
  publisher = "Elsevier BV",
  volume    =  101,
  number    =  2,
  pages     = "185--232",
  month     =  oct,
  year      =  2002,
  doi       = "10.1016/s0304-4149(02)00150-3",
  issn      = "0304-4149,1879-209X",
  language  = "en"
}

@article{higham2002strong,
  title     = "{Strong convergence of {E}uler-type methods for nonlinear
               stochastic differential equations}",
  author    = "Higham, Desmond J and Mao, Xuerong and Stuart, Andrew M",
  journal   = "SIAM J. Numer. Anal.",
  publisher = "Society for Industrial \& Applied Mathematics (SIAM)",
  volume    =  40,
  number    =  3,
  pages     = "1041--1063",
  month     =  jan,
  year      =  2002,
  doi       = "10.1137/s0036142901389530",
  issn      = "0036-1429,1095-7170"
}

@BOOK{Stroock2013-df,
  title     = "{Multidimensional Diffusion Processes}",
  author    = "Stroock, Daniel W and Varadhan, S R S",
  publisher = "Springer",
  address   = "Berlin, Germany",
  series    = "Classics in Mathematics",
  month     =  jul,
  year      =  2013,
  doi       = "10.1007/3-540-28999-2",
  isbn      = "9783662222010,9783540289999",
  language  = "en"
}

@INCOLLECTION{Bellet2006-pq,
  title     = "{Ergodic Properties of Markov Processes}",
  author    = "Bellet, Luc Rey",
  editor    = "Attal, Stéphane and Joye, Alain and Pillet, Claude-Alain",
  booktitle = "{Open Quantum Systems II: The Markovian Approach}",
  publisher = "Springer Berlin Heidelberg",
  address   = "Berlin, Heidelberg",
  pages     = "1--39",
  year      =  2006,
  doi       = "10.1007/3-540-33966-3\_1",
  isbn      =  9783540339663
}

@article{pereyra2016proximal,
  title   = "{Proximal {M}arkov chain {M}onte {C}arlo algorithms}",
  author  = "Pereyra, Marcelo",
  journal = "Stat. Comput.",
  volume  =  26,
  number  =  4,
  pages   = "745--760",
  month   =  jul,
  year    =  2016,
  doi     = "10.1007/s11222-015-9567-4",
  issn    = "0960-3174,1573-1375"
}

@article{durmus2018efficient,

  title     = "{Efficient {B}ayesian computation by proximal {M}arkov chain {M}onte
               {C}arlo: When Langevin meets {M}oreau}",
  author    = "Durmus, Alain and Moulines, Eric and Pereyra, Marcelo",
  journal   = "SIAM J. Imaging Sci.",
  publisher = "Society for Industrial \& Applied Mathematics (SIAM)",
  volume    =  11,
  number    =  1,
  pages     = "473--506",
  month     =  jan,
  year      =  2018,
  doi       = "10.1137/16m1108340",
  issn      = "1936-4954",
  language  = "en"
}

@inproceedings{zhang2020wasserstein,
  title={Wasserstein control of mirror {L}angevin {M}onte {C}arlo},
  author={Zhang, Kelvin Shuangjian and Peyr{\'e}, Gabriel and Fadili, Jalal and Pereyra, Marcelo},
  booktitle={Conference on Learning Theory},
  pages={3814--3841},
  year={2020},
  organization={PMLR},
doi={ 
10.48550/arXiv.2002.04363},
}

@article{girolami2011riemann,
  title     = "{Riemann manifold {L}angevin and {H}amiltonian {M}onte {C}arlo methods}",
  author    = "Girolami, Mark and Calderhead, Ben",
  journal   = "J. R. Stat. Soc. Series B Stat. Methodol.",
  publisher = "Blackwell Publishing Ltd",
  volume    =  73,
  number    =  2,
  pages     = "123--214",
  month     =  mar,
  year      =  2011,
  doi       = "10.1111/j.1467-9868.2010.00765.x",
  issn      = "1369-7412,1467-9868"
}

@article{west1987scale,
  title={On scale mixtures of normal distributions},
  author={West, Mike},
  journal={Biometrika},
  volume={74},
  number={3},
  pages={646--648},
  year={1987},
  publisher={Oxford University Press},
  doi       = "10.1093/biomet/74.3.646"
}

@article{park2008bayesian,
  title={The {B}ayesian {L}asso},
  author={Park, Trevor and Casella, George},
  journal={Journal of the American Statistical Association},
  volume={103},
  number={482},
  pages={681--686},
  year={2008},
  publisher={Taylor \& Francis},
doi={10.1198/016214508000000337}
}

@ARTICLE{Dereich2011-tp,
  title     = "{An Euler-type method for the strong approximation of the
               Cox--Ingersoll--Ross process}",
  author    = "Dereich, Steffen and Neuenkirch, Andreas and Szpruch, Lukasz",
  journal   = "Proceedings of the Royal Society A: Mathematical, Physical and
               Engineering Sciences",
  publisher = "Royal Society",
  volume    =  468,
  number    =  2140,
  pages     = "1105--1115",
  month     =  dec,
  year      =  2011,
  doi       = "10.1098/rspa.2011.0505"
}

@ARTICLE{Lelievre2016-hx,
  title     = "{Partial differential equations and stochastic methods in
               molecular dynamics}",
  author    = "Leli{\`e}vre, Tony and Stoltz, Gabriel",
  journal   = "Acta Numer.",
  publisher = "Cambridge University Press",
  volume    =  25,
  pages     = "681--880",
  month     =  may,
  year      =  2016,
  issn      = "0962-4929, 1474-0508",
  doi       = "10.1017/S0962492916000039"
}

@BOOK{Meyn2012-oq,
  title     = "{Markov Chains and Stochastic Stability}",
  author    = "Meyn, Sean P and Tweedie, Richard L",
  publisher = "Springer Science \& Business Media",
  month     =  dec,
  year      =  2012,
  language  = "en",
  isbn      = "9781447132677",
doi="10.1007/978-1-4471-3267-7"
}

@ARTICLE{Markowich2004-dw,
  title={On the trend to equilibrium for the {F}okker--{P}lanck equation: an interplay between physics and functional analysis},
  author={Markowich, Peter A and Villani, C{\'e}dric},
  journal={Mat. Contemp},
  volume={19},
  pages={1--29},
  year={2000},
url={https://mc.sbm.org.br/wp-content/uploads/sites/9/sites/9/2021/12/19-1.pdf}
}

@BOOK{Karatzas1991-xr,
  title     = "{Brownian Motion and Stochastic Calculus}",
  author    = "Karatzas, Ioannis and Shreve, Steven E",
  publisher = "Springer Science \& Business Media",
  year      =  1991,
  isbn      = "9780387976556",
doi="10.1007/978-1-4684-0302-2"
}

\appendix

\renewcommand{\theequation}{SM\arabic{equation}}
\section{Ergodicity under bounded assumptions}\label{app-sec-ergod}

Consider two measures $\pi_1,\pi_2$ such that $\pi_1$ is absolutely continuous with respect to $\pi_2$ and denote the Radon--Nikodym derivative by $d\pi_1/d\pi_2$. Define the entropy of $\pi_1$ with respect to $\pi_2$ (also called the Kullback--Leibler or KL divergence) by $\mathcal H(\pi_1\mid \pi_2)=\mathbb{E}_{\pi_1}(\log(d\pi_1/d\pi_2))$ and the Fisher information $\mathcal I(\pi_1\mid \pi_2)=\mathbb{E}_{\pi_1}\bp{\norm{\nabla \log(d\pi_1/d\pi_2)}^2}$. We say $\pi_2$ satisfies a log-Sobolev inequality LSI$(R)$ if, for a constant $R>0$,
\[\mathcal{H}(\pi_1\mid \pi_2)\le\frac{1}{2\,R} \mathcal{I}(\pi_1\mid \pi_2),\qquad \forall \pi_1\ll \pi_2. 
\]
The theory of log Sobolev inequalities (e.g., \cite{Lelievre2016-hx}) gives the following geometric convergence of $\rho(t,\cdot)$ to the invariant measure $\pi$. 
\begin{theorem} \label{thm_bndd_G} Assume that $G$ is bounded.
    The solution $(u,v)_t$ of \cref{main} with initial distribution $\rho_0$ converges to the invariant measure $\pi$ geometrically. In particular, 
    \begin{equation}\mathcal H(\rho(t,\cdot) \mid \pi)\le \exp(-2\,R\,t/\beta)\, \mathcal H(\rho_0 \mid \pi),\quad t>0,
    \label{hhh}\end{equation}
    for $R=\beta\, \lambda\,e^{\inf G-\sup G}$.
\end{theorem}
The assumption that $G$ is bounded holds for  $G(x)=\sum \sigma(y_i a_i^\top x)$ when $\sigma$ is bounded such as the smooth approximations of the unit loss that are mentioned in the introduction.
\begin{proof}
 In terms of the density $p(t,\cdot)$ for \cref{cyl} and the Cartesian co-ordinates $\vec x_t\in \real^{3d}$, standard theory (Bakry--Emery criterion, Holley--Stroock theorem) applies to the potential $\frac 12 \lambda \|\vec x \|^2+G_3(\vec x)$, as it is a bounded perturbation (given by $G$) of a convex potential (given by $\frac12\lambda\|\vec x\|^2$).
 See~\cite[Proposition 2.9]{Lelievre2016-hx}. This gives $\mathcal H(p(t,\cdot) \mid \pi)\le \exp(-2\,R\,t/\beta)$ for $R$ defined above. Taking a cylindrical initial distribution, $p(t,\cdot)=\rho(t,\cdot)$, we have \cref{hhh}.
\end{proof}

\section{Link to Riemannian Langevin}\label{app-sec-riemann-geom}
\begin{proof}[Proof of Proposition \ref{prop-riemann-geom}]
By It\^o's formula,
$$
d\mathbf{z}_t  = \underbrace{ \begin{bmatrix}
 -2 \lambda \eta_t - \frac12 x_t \odot \nabla G(x_t) + 1 \\[1.05em]
- \pp{ 4\eta_t +\frac{ x_t^2}{4\eta_t }} \odot \nabla G(x_t) -  2  \lambda x_t +  \frac{x_t}{4 \eta_t} 
\end{bmatrix}}_{\mu(\mathbf{z}_t)} dt +  \sqrt{2}   \underbrace{ \begin{bmatrix}
\frac12 \diag(u_t)&0 \\
\diag(v_t)	& \diag(u_t)
\end{bmatrix}  }_{S(\mathbf{z}_t)} \binom{dW^1_t}{dW^2_t} .
$$
A direct computation gives
$$
-M(\mathbf{z})\nabla H(\mathbf{z}) = \begin{bmatrix}
    -2\lambda \eta - \frac12 x \odot \nabla G(x) - \frac12\\
    - \pa{4\eta + \frac{x^2}{4\eta} }\odot \nabla G(x) -2\lambda x - \frac{x}{4\eta}
\end{bmatrix}\quad \text{and}\quad  \Gamma(\mathbf{z}) =\frac12 \begin{bmatrix}
    3 \cdot 1_d\\ x/\eta
\end{bmatrix},
$$
where $\mathfrak{g}(\mathbf{z}) = M(\mathbf{z})^{-1}$ is the inverse metric and $\Gamma$ is the correction drift defined in \eqref{eq:riemannian}. Therefore $\mu(\mathbf{z}) = -\mathfrak{g}(\mathbf{z})^{-1}\nabla H(\mathbf{z}) + \Gamma(\mathbf{z})$, and $M(\mathbf{z}) = S(\mathbf{z}) S(\mathbf{z})^\top$.

Letting $\rho(t,\cdot)$ denote the distribution of $\mathbf{z}_t$, the Fokker--Planck equation reads
$$
\frac{\partial}{\partial t} \rho = - \mathrm{div}\pa{ \mu^\top(\mathbf{z}) \rho }  + \sum_{i=1}^{2d} \sum_{j=1}^{2d} \frac{\partial^2}{\partial \mathbf{z}_i\, \partial \mathbf{z}_j} \pa{ M_{i,j}(\mathbf{z}) \rho  }.
$$
Expanding the second term gives
\begin{align*}
\sum_{i=1}^{2d} \sum_{j=1}^{2d} \frac{\partial^2}{\partial \mathbf{z}_i\, \partial \mathbf{z}_j} \pa{ M_{i,j}(\mathbf{z}) \rho  } 
&= \mathrm{div}\pa{ M \nabla \rho } + \mathrm{div}\pa{ \Bigl[ \sum_{j} \frac{\partial}{\partial \mathbf{z}_j} M_{i,j}(\mathbf{z})\, \rho\Bigr]_{i=1,\ldots,2d} }\\
&= \mathrm{div}\pa{ M \nabla \rho } + \mathrm{div}\pa{ \binom{ 3/2 }{ \tfrac{ x}{2\eta}  } \rho   }.
\end{align*}
Combining $-\mu^\top$ with this correction, we obtain
\begin{align*}
\frac{\partial}{\partial t} \rho &= \mathrm{div}\pa{   \begin{bmatrix}
 2\lambda\eta + \tfrac12 x\odot \nabla G(x) + \tfrac12 \\[0.3em]
 \pa{ 4\eta + \tfrac{ x^2}{4\eta} }\odot \nabla G(x) +  2\lambda x +  \tfrac{x}{4\eta}  
\end{bmatrix} \rho  } + \mathrm{div}\pa{ M \nabla \rho } \\
&=\mathrm{div} \pa{ \begin{pmatrix}
\eta &\tfrac12 x   \\
\tfrac 12 x &\tfrac{ x^2}{4\eta} +4 \eta 
\end{pmatrix}
\pa{ \binom
{  2\lambda - \lambda\tfrac{ x^2}{8\eta^2} + \tfrac{1}{2\eta} }
{ \nabla G(x) + \lambda  \tfrac{ x}{4\eta}  }
+ \nabla \log \rho } \rho} \\
&=  \mathrm{div}\pa{M(\mathbf{z})\, \nabla\pa{H(\mathbf{z})+ \log(\rho)}\, \rho }.\qedhere
\end{align*}
\end{proof}

\section{Additional proofs for the discretized scheme}\label{sec-add-proofs-discr-scheme}

We first remark that our numerical scheme \eqref{discrete_scheme} can be written explicitly as
\begin{gather}\begin{split}
    \binom{u^{k+\frac12}}{v^{k+\frac12}} &= \binom{u^k}{v^k} - \deltat \,\binom{v^k \odot \nabla G(u^k\odot v^k)}{u^k \odot \nabla G(u^k\odot v^k)}  + \sqrt{\frac{2}{\beta}}
    \binom{\deltaw^1_k}{\deltaw^2_k},\\
    u^{k+1} &=  \frac{ u^{k+\frac12} + \sqrt{ u^{k+\frac12}\odot u^{k+\frac 12} + 4(\tstep/\beta) (1+\tstep\lambda) }}{2(1+\tstep\lambda)},\\
    v^{k+1} &= \frac{1}{1+\tstep\lambda} v^{k+\frac12}.\end{split}\label{themethod}
\end{gather}
 where the square root is interpreted element-wise and we have used the fact that, for $u>0$,
 \begin{align*}
 u = \tstep \,\frac{1}{\beta\,u} - \tstep\, \lambda\, u + w &\iff (1+\tstep \lambda) u^2  - w\, u -\tstep/\beta  = 0\\ &\iff u =\frac{ w + \sqrt{ w^2 + 4(\tstep/\beta) (1+\tstep \lambda) }}{2(1+\tstep \lambda)}.
 \end{align*}
 It is straightforward to see from this that positivity of $u^k$ is preserved.


\subsection{Explicit form of the transition density}
\begin{lemma}\label{lem:transition_densities}
Up to constants, the log transition kernel $P$ for  moving from $x = (u,v)^k$  to $\hat x= (u,v)^{k+1}$ under \cref{discrete_scheme} is
\begin{align*}
    \log(P)(x, \hat x) &= -\frac{\beta}{4\tstep} \norm{   (1+\tstep\,\lambda) u^{k+1} - \frac{\tstep}{\beta u^{k+1}} - u^k  + \tstep\, v^k\odot \nabla G(u^k\odot v^k) ) }^2\\
    &\quad -\frac{\beta}{4\tstep} \norm{   (1+\tstep\,\lambda) v^{k}  - v^{k+1}  + \tstep\, u^k\odot \nabla G(u^k\odot v^k)) ) }^2\\
    &\quad + \sum_{i=1}^d\log\pa{ 1+\tstep\,\lambda + \frac{\tstep}{\beta  (u^k_i)^2}}.
\end{align*}
\end{lemma}
\begin{proof}
Let
$x = (u,v)$  and $g(x) = (v\odot \nabla G(u\odot v), u\odot \nabla G(u\odot v))$. From \cref{discrete_scheme}, we have $P(x) = P_1(P_0(x))$ for
$$
P_0(x) = x - \tstep \,g(x) + \sqrt{\frac{2\tstep}{\beta} }\Nn(0,\Id_{2d})
$$
and $P_1(x)  = (P_{1,u}(x), P_{1,v}(x))$ with 
$$
P_{1,u}(x) = \frac{u + \sqrt{u\odot u + \frac{4}{\beta} \tstep (1+\tstep \lambda))}}{2(1+\tstep\lambda)}, \qquad P_{1,v}(x) = \frac{v}{(1+\tstep\lambda)}.
$$
Note that $ P_0(x) \sim\Nn\pa{x - \tstep \,g(x) , (2\,\tstep/\beta)\,\Id_{2d}  }$  is Gaussian with transition density
$$
f_0(z) \propto \exp\pa{ -\beta\frac{(z-(x-\tstep\, g(x)))^2}{4\,\tstep} }
$$
To work out the density of $u = P_{1,u}(z)$ with $z=P_0(x)$, write the  cumulative distribution function $F_{1,u}(t)$ as
\begin{align*}
    F_{1,u}(t) &= \PP(u\leq t) = \PP\pa{\frac{z + \sqrt{z\odot z + \frac{4}{\beta}\tstep\, (1+\tstep \,\lambda))}}{2(1+\tstep\,\lambda)} \leq t }\\
    &= \PP\pa{   z   \leq (1+\tstep\,\lambda) t - \frac{\tstep }{t\,\beta }}.
\end{align*}
It follows that the density  of $Y_u$ is
$$
f_{1,u}(t) = \frac{d}{dt}F_{1,u}(t) \propto f_0\pa{ (1+\tstep\,\lambda)\, t - \frac{\tstep}{\beta \,t} }  \pa{(1+\tstep\,\lambda) +  \frac{\tstep}{\beta \,t^2} }.
$$
The density of $P_{1,v}(v)$ is
$f_{1,v}(t)\propto f_0( (1+\tstep\,\lambda)\,t)$.
\end{proof}

\subsection{Proof of drift condition of order 4}\label{supp_drift}
\begin{lemma}\label{lem:drift_4th}
    Let $V_4((u,v)) = 1+ \norm{(u,v)}^4$. There exists $C,c,\nu>0$ such that for all $\deltat<c\lambda$,
    $$
    \EE[V_4((u,v)^{k+1})\mid (u,v)^k] \leq (1-\nu\,\tstep)\, V_4((u,v)^k) + C\, \deltat.
    $$    
\end{lemma}
\begin{proof}[Proof of Lemma \ref{lem:drift_4th}]
    Since
\begin{align}
    (1+\lambda\,\deltat )u^{k+1} - \frac{\deltat}{\beta\, u^{k+1}} = u^k - v^k \odot \nabla G(u^k \odot v^k) \deltat + \sqrt{\frac{2}{\beta}}\Delta W^1_k,
\end{align}
taking the square norm yields
\begin{align}
    (1+\lambda\deltat )^2\snorm{u^{k+1}}^2 - \frac{2\,d\,\deltat\,(1+\lambda\,\deltat)}{\beta }& \leq  \norm{u^k - v^k \odot \nabla G(u^k \odot v^k) \deltat}^2 + \frac{2}{\beta}\,\snorm{\Delta W^1_k}^2 \\
    &\qquad+ \sqrt{\frac{4}{\beta}}\dotp{\Delta W^1_k}{u^k - v^k \odot \nabla G(u^k \odot v^k) \deltat}.
\end{align}
Squaring this again and taking expectation conditional on $u^k,v^k$ gives
\begin{align*}
    &(1+\lambda\deltat )^4\EE[\snorm{u^{k+1}}^4\mid (u,v)^k] - \frac{4d\deltat(1+\lambda\deltat)^3}{\beta } \EE[\snorm{u^{k+1}}^2\mid(u,v)^k]\\
    & \leq  \norm{u^k - v^k \odot \nabla G(u^k \odot v^k) \deltat}^4 + \frac{4}{\beta^2}\EE\snorm{\Delta W^1_k}^4
    +\frac{4}{\beta} \snorm{u^k - v^k \odot \nabla G(u^k \odot v^k) \deltat}^2  \EE\snorm{\Delta W^1_k}^2 
    \\
    &\qquad+ \frac{4}{\beta}\EE\snorm{\Delta W^1_k}^2\norm{u^k - v^k \odot \nabla G(u^k \odot v^k) \deltat}^2.
\end{align*}
We now use the fact that $\EE_k \norm{\Delta W_k^1}^2 = d \,\tstep$ and $\EE_k \norm{\Delta W_k^1}^4 = (d^2+2d)\,\tstep^2$ to obtain
\begin{align*}
    &(1+\lambda\deltat )^4\EE[\snorm{u^{k+1}}^4\mid (u,v)^k] - \frac{4d\deltat(1+\lambda\deltat)^3}{\beta } \EE[\snorm{u^{k+1}}^2\mid (u,v)^k]\\
    & \leq  \norm{u^k - v^k \odot \nabla G(u^k \odot v^k) \deltat}^4 + \frac{4 (d^2+2d)\deltat^2}{\beta^2}
    +\frac{8d\deltat}{\beta} \norm{u^k - v^k \odot \nabla G(u^k \odot v^k) \deltat}^2 .
\end{align*}
Applying Young's inequality for products gives that, for any $\alpha_1,\alpha_2>0$,
\begin{align*}
    &\pa{(1+\lambda\deltat )^4-\alpha_1 \deltat}\EE[\snorm{u^{k+1}}^4\mid(u,v)^k] \\
    &\leq (1+\alpha_2\deltat) \norm{u^k - v^k \odot \nabla G(u^k \odot v^k) \deltat}^4 + \frac{4 (d^2+2d)\deltat^2}{\beta^2}
    + \frac{64 d^2\deltat}{\beta^2 \alpha_2} + \frac{4 d^2\deltat(1+\lambda\deltat)^6}{\alpha_1\beta^2 }.
\end{align*}

As in the proof of \cref{lem:2ndmoments_discrete}, we have
\begin{align*}
    \norm{u^k - v^k \odot \nabla G(u^k \odot v^k) \deltat}^2 \leq \snorm{u^k}^2 + \deltat^2 B^2 \snorm{v^k}^2 + 2K \deltat
\end{align*}
and squaring this gives
\begin{align*}
    &\norm{u^k - v^k \odot \nabla G(u^k \odot v^k) \deltat}^4 \\
    &\leq  \snorm{u^k}^4 + \deltat^4 B^4\snorm{v^k}^4 +2\snorm{u^k}^2  \deltat^2 B^2 \snorm{v^k}^2+  4K \deltat \snorm{u^k}^2 +4K \deltat^3 B^2 \snorm{v^k}^2 + 4K^2 \deltat^2\\
    &\leq
    (1+\deltat^2 B^2)\snorm{u^k}^4 + (B^2\deltat^2 + \deltat^4 B^4)\snorm{v^k}^4 
    \\
    &\qquad+\alpha_3\deltat \snorm{u^k}^4 + \frac{4K^2\deltat}{\alpha_3 }+\alpha_3\deltat^2\snorm{v^k}^4 +\frac{4K^2\deltat^4B^4}{\alpha_3} +4K^2\deltat^2\\
    &=
    (1+\alpha_3\deltat +\deltat^2 B^2)\snorm{u^k}^4 + \deltat^2(\alpha_3 + B^2 + \deltat^2 B^4)\snorm{v^k}^4 +4K^2\deltat(\alpha_3^{-1} + \deltat^3B^4\alpha_3^{-1}+\deltat).
\end{align*}
It follows that
\begin{align}
   \pa{(1+\lambda\deltat )^4-\alpha_1 \deltat} \EE\bp{\snorm{u^{k+1}}^4\mid (u,v)^k} \leq 
    c_1\snorm{u^k}^4 + c_2\deltat^2 \snorm{v^k}^4 +c_3\deltat,
\end{align}
where
$$
c_1 = (1+\alpha_2 \deltat)(1+\alpha_3\deltat +\deltat^2 B^2),\qquad
c_2 = (1+\alpha_2 \deltat)(\alpha_3 + B^2 + \deltat^2 B^4),
$$
and
$$
c_3 =
(1+\alpha_2 \deltat) 4K^2(\alpha_3^{-1} + \deltat^3B^4\alpha_3^{-1}+\deltat) 
+\frac{64 d^2}{\beta^2 \alpha_2} + \frac{4 d^2(1+\lambda\deltat)^6}{\alpha_1\beta^2 }.
$$
One can carry out a similar computation for $v^{k+1}$ to obtain
$$
(1+\lambda\deltat)^4 \mathbb{E}[\snorm{v^{k+1}}^4 \mid (u,v)^k]\leq  c_1 \snorm{v^k}^4 +  c_2 \deltat^2 \snorm{u^k}^4 + \tilde c_3\deltat,
$$
with
$$
\tilde c_3 =
(1+\alpha_2 \deltat) 4K^2(\alpha_3^{-1} + \deltat^3B^4\alpha_3^{-1}+\deltat) 
+\frac{64 d^2}{\beta^2 \alpha_2}.
$$
It follows that 
\begin{align}
   ((1+\lambda\deltat)^4-\alpha_1\deltat)\, \EE[\snorm{(u,v)^{k+1}}^4 \mid (u,v)^k]\leq (c_1 +c_2\deltat^2) \snorm{(u,v)^{k}}^4 + 2c_3\deltat.
\end{align}
Finally, define
\begin{align}
   \alpha\eqdef  \frac{c_1+c_2\deltat^2}{(1+\lambda\deltat)^4-\alpha_1\deltat}.
\end{align}
Then, $\alpha<1$ provided that
$$
\alpha_1 \deltat + (1+\alpha_2 \deltat)(1+\alpha_3\deltat +\deltat^2 B^2)
+\deltat^2 (1+\alpha_2 \deltat)(\alpha_3 + B^2 + \deltat^2 B^4) < (1+\lambda \deltat)^4 .
$$
The leading exponent to $\deltat^0$ on both sides is 1, and the leading exponent to $\deltat$ is $4\lambda$ and $\alpha_1 + \alpha_2+\alpha_3$ on the right and left and side respectively. By choosing
$\alpha_1,\alpha_2,\alpha_3\lesssim \lambda $ and $\deltat\lesssim \lambda$, we have $\alpha<1$. In other words, we may choose $\nu>0$ and write $\alpha \le 1-\nu \tstep$ for $\tstep$ sufficiently small.
\end{proof}

\section{Numerical experiments}\label{app-numerics}

The code for reproducing our numerical experiments can be found online\footnote{ \url{https://hadamardlangevin.readthedocs.io/}}.
Throughout, we consider the setting where $G(x) = \|Ax-y\|^2/2$. The stepsize and Moreau-envelope regularization choice for Prox-l1 are taken to be $\gamma=1/(KL)$ for some $K\geq 1$ and $\deltat = \gamma/(5(\gamma L +1))$, where $L= \norm{A}^2$ is the Lipschitz constant of the gradient of $G$. This is consistent with the parameter choices recommended in \cite{durmus2018efficient}.

\subsection{Synthetic  experiment}
We consider the case of $d=20$, and $A\in\RR^{40\times 20}$ is a random Gaussian matrix of mean zero and variance $\frac{1}{16m}$ for $m=40$. We let $\lambda = \frac12 \norm{A^\top y}_\infty$ and $y=A x_0$, where $x_0$ has exactly 2 non-zero entries.  For the Prox-l1 and Hadamard--Langevin samplers, we run $10^4$ burn-in iterations before recording $10^5$  samples,  while for the Gibbs sampler (due to its slower running time and faster convergence), we record $10^4$ samples after $10$ burn-in samples. To assess the convergence of each scheme, we evaluate the effective sample size (ESS), which is computed using the package arviz \citep{arviz_2019}.   The sample means for each of the 20 dimensions are shown in   \cref{fig:means-dim20}, which also shows the ESS in each dimension. The Gibbs sampler has the highest ESS, followed by Hadamard--Langevin. In general, for these low-dimensional settings, the  Gibbs sampler displays fast mean convergence behavior, and the Hadamard--Langevin sampler also converges well (although less sharp than the Gibbs sampler).

\begin{figure}[!htp]
    \centering
    \begin{tabular}{c@{}c@{}c@{}c}
    Prox-l1&Hadamard&Gibbs& ESS\\
\includegraphics[width=0.24\linewidth]{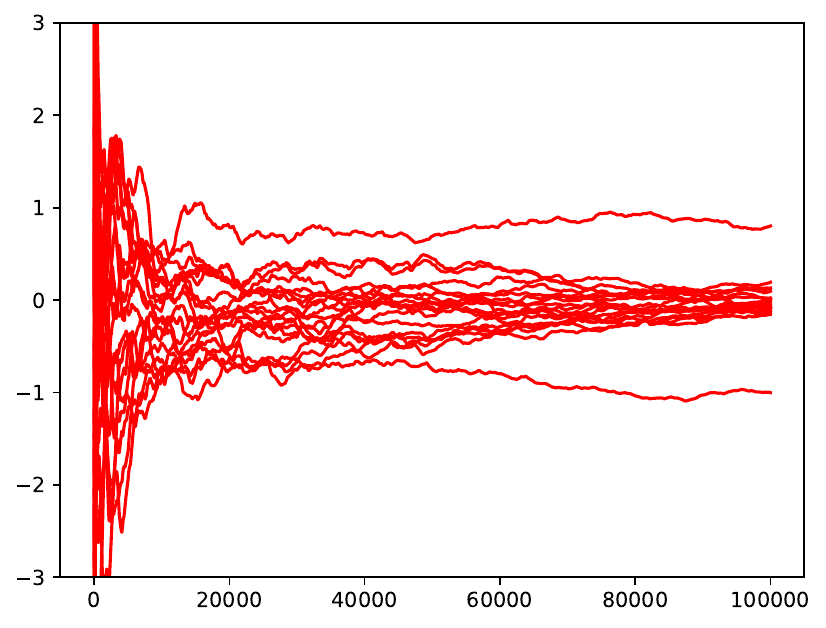}
&
\includegraphics[width=0.24\linewidth]{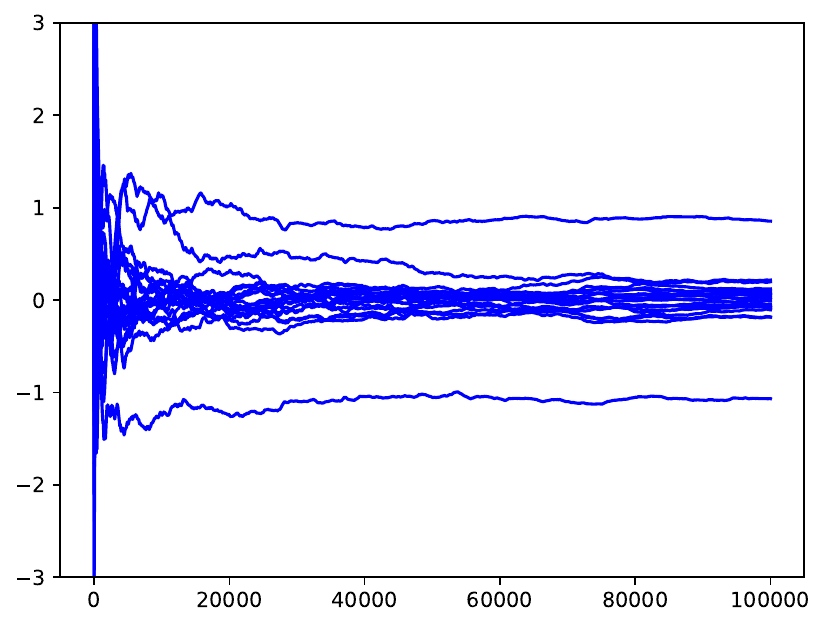}
&
\includegraphics[width=0.24\linewidth]{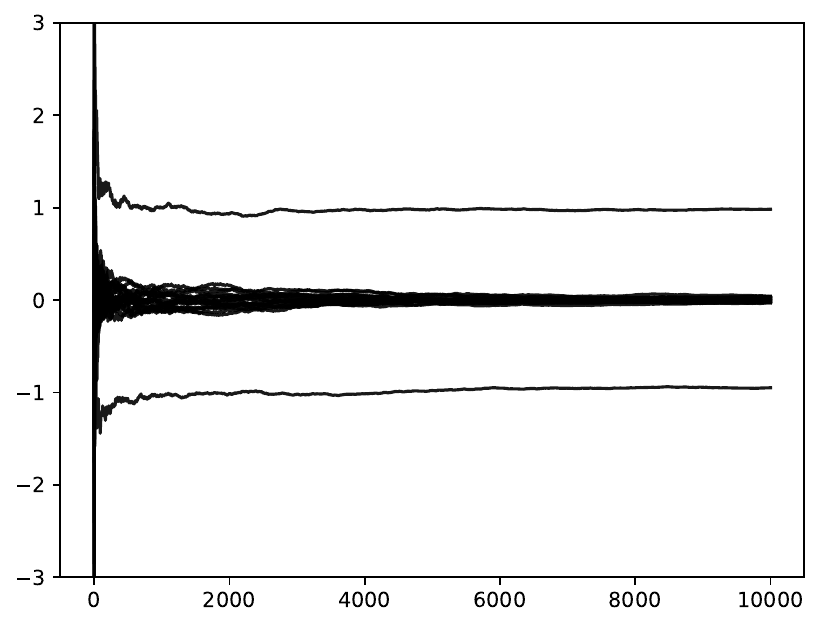}
&\includegraphics[width=0.24\linewidth]{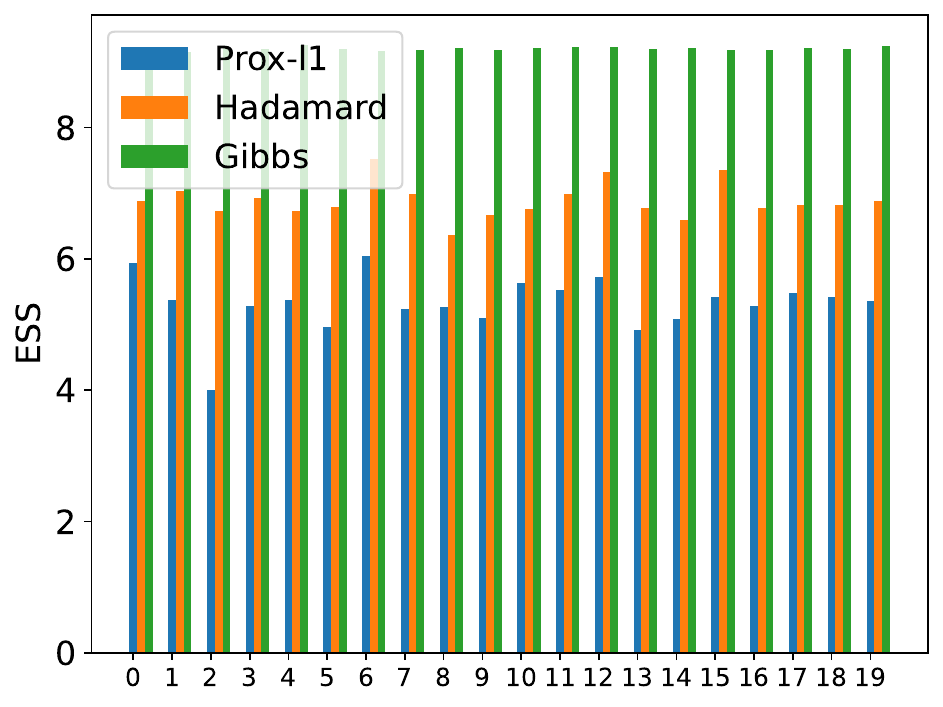}
    \end{tabular}
    \caption{Plots of sample means against iteration. The right figure shows the ESS in log-scale for each of the 20 dimensions. The minimum ESS (in log scale) for Prox-l1 is 4.0, for Hadamard is 6.4 and for Gibbs is 9.2.}
    \label{fig:means-dim20}
\end{figure}

\subsection{Deconvolution with Haar wavelets}
We consider Gaussian deconvolution of a  piecewise-constant signal (length $d=1024$). The matrix $A = U W^\top$ where $U$ is a convolution operator and $W$ is an orthogonal Haar transform. The regularization parameter is $\lambda =1$. We run Proximal-Langevin and Hadamard--Langevin samplers at stepsize $\tau=0.01$,  with $10^2$ burn-in iterations and record $10^5$ samples. The sample means   and ESS scores are shown in Figure \ref{fig:haar}. The ground truth is displayed as x0. In log scale, we also plot the difference between the  95th and 5th quantiles. Notice that both methods have a higher difference around the jumps, which naturally aligns with the fact that reconstruction of the precise jump locations is more unstable.  

\subsection{Image inpainting}
We perform  inpainting: let $A = UW^\top$, where $U$ is a masking operator and $W$ is the wavelet transform with the Haar wavelet, with $\beta=5,100,500$. For both Prox-l1 and Hadamard--Langevin samplers, we run 10,000 burn-in iterations, then record every 20 samples out of 6000 samples. On the $300$ recorded samples, we compute the mean images and quantile differences.   The corrupted input as well as the mode image (i.e., the MAP estimate)  are shown in \cref{fig:inpainting1} in the case of $\beta=5$. The mean images and difference in 95th and 5th quantiles are shown in \cref{fig:inpainting2}. For small $\beta$, the largest variations in the quantiles for both methods occur where the obscured pixels are. By increasing the inverse temperature $\beta$,  the distribution concentrates more around the mode. In this case, we observe that the largest variation occurs around the image edges, which suggests that the uncertainty as $\beta$ increases is more around the location of the image edges.
The observations of this latter experiment are in line with the observations of \cite{durmus2018efficient}.

\begin{figure}
    \centering
\begin{tabular}{c@{}c@{}c}
(a)&(b)&(c)\\
     \includegraphics[width=0.32\linewidth]{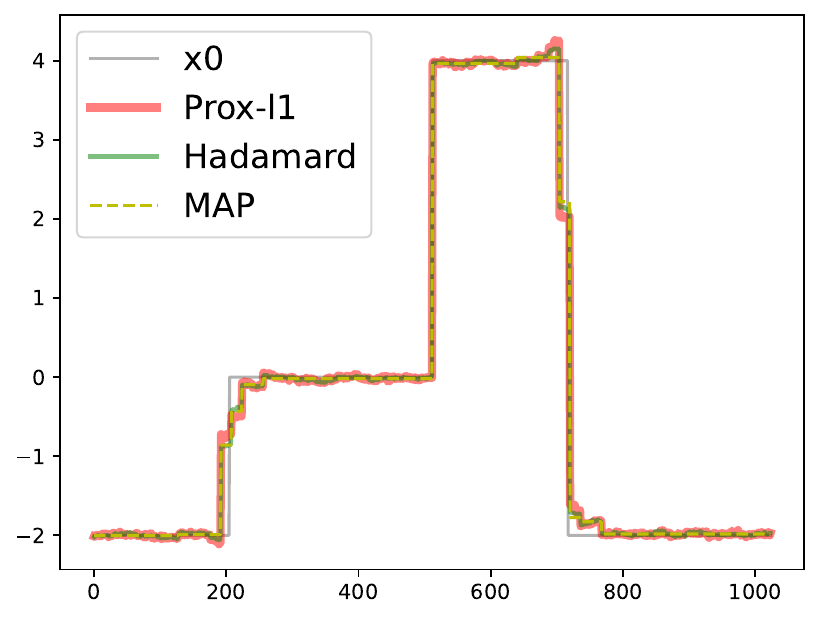}
&\includegraphics[width=0.32\linewidth]{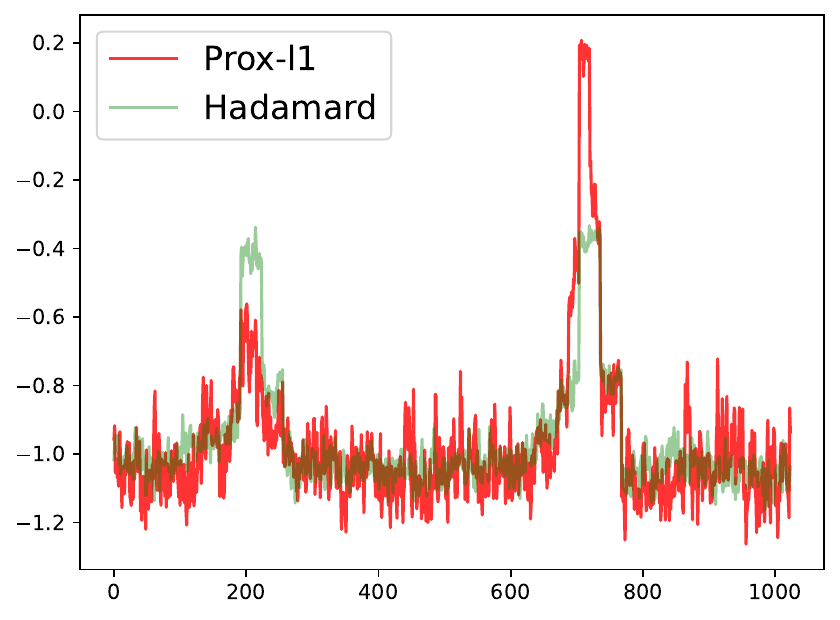}
&
\includegraphics[width=0.32\linewidth]{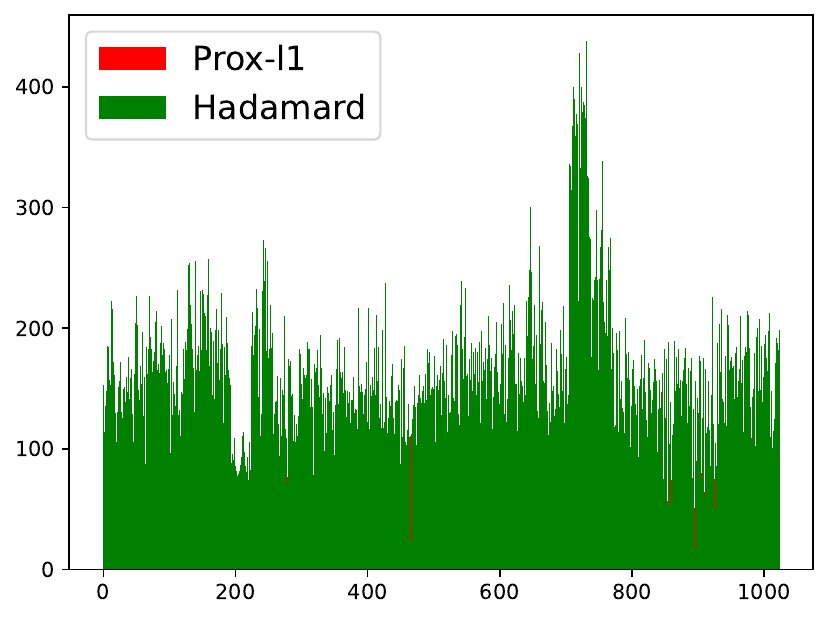}
\end{tabular}
    \caption{1d Haar wavelet deconvolution: (a) the mode and mean signals, (b)  the difference between the 95th  and 5th quantiles, and (c)  the  ESS at each pixel. The  ESS across all pixels is higher for Hadamard--Langevin dynamics than Prox-l1. }
    \label{fig:haar}
\end{figure}

\begin{figure}
\centering
\begin{tabular}{c@{}c@{}}
Observation&Mode\\
\includegraphics[width=0.3\linewidth]{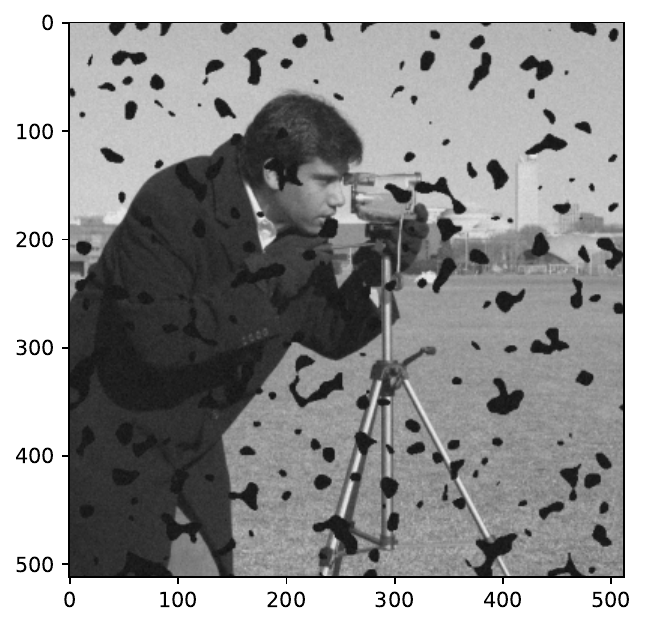}
&
\includegraphics[width=0.3\linewidth]{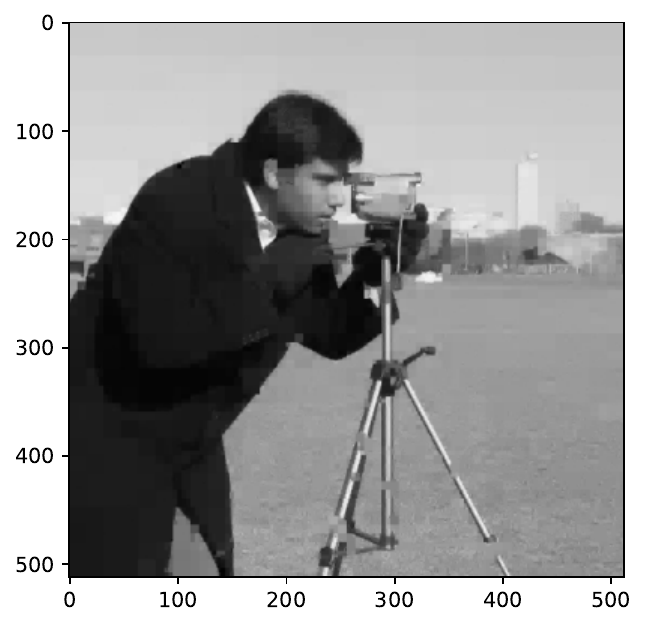}
     \end{tabular}
         \caption{Observation and mode image.}
\label{fig:inpainting1}
\end{figure}

\begin{figure}[!htp]
    \centering

\begin{tabular}{c@{}c@{}c@{}c@{}c}
&Hadamard&Hadamard&Prox-l1&Prox-l1\\    
&(mean)&(percentile)&(mean)&(percentile)\\
\rotatebox{90}{\hspace{2em}$\beta=5$}&
\includegraphics[width=0.21\linewidth]{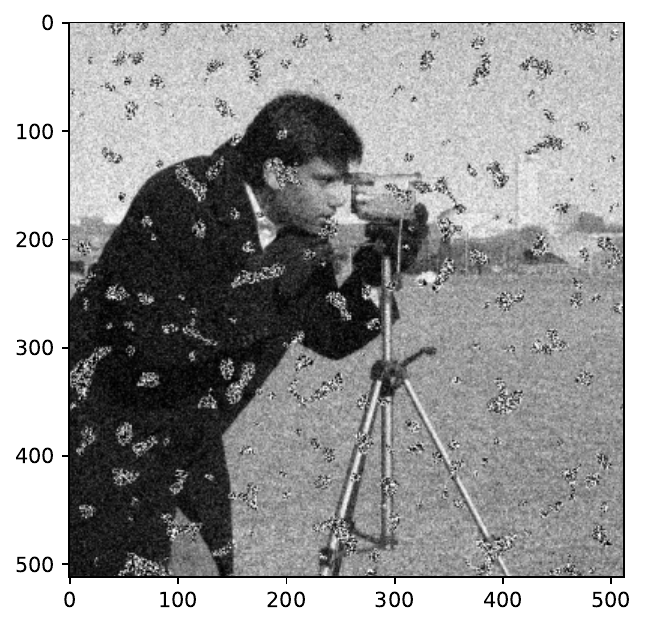}
&\includegraphics[width=0.25\linewidth]{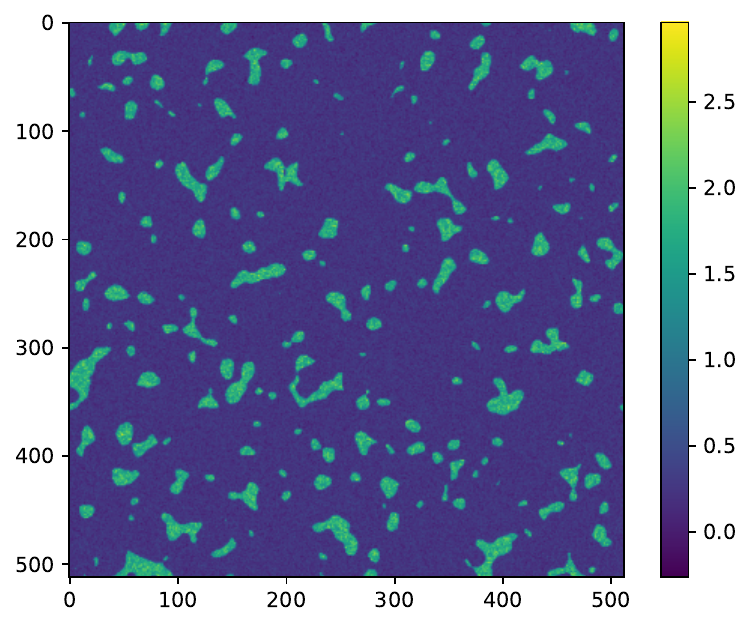}
&
\includegraphics[width=0.21\linewidth]{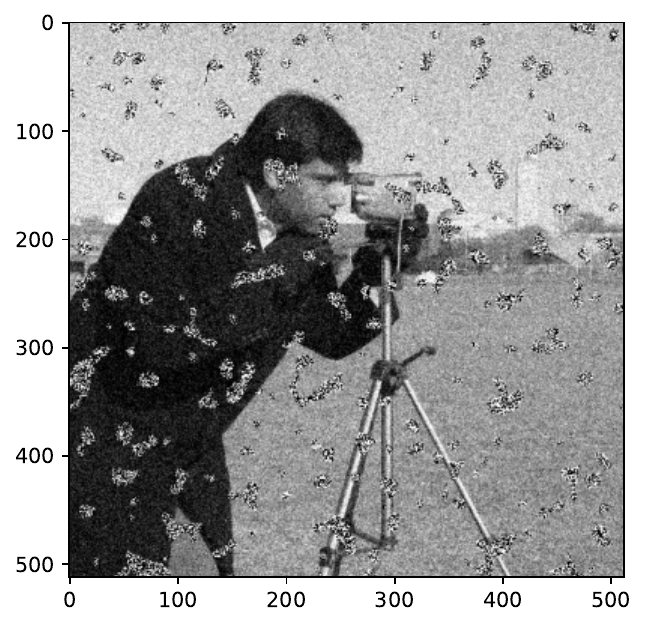}
&
\includegraphics[width=0.25\linewidth]{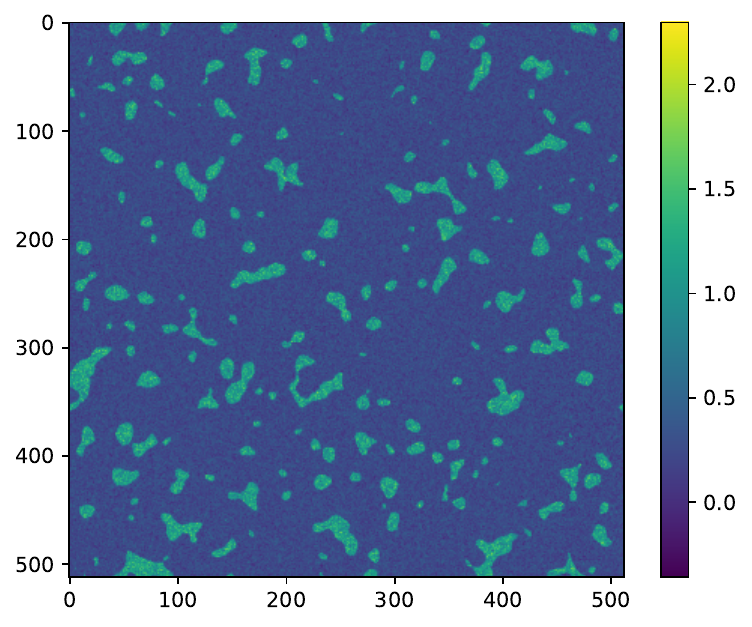}
\\
\rotatebox{90}{\hspace{2em}$\beta=100$}&
\includegraphics[width=0.21\linewidth]{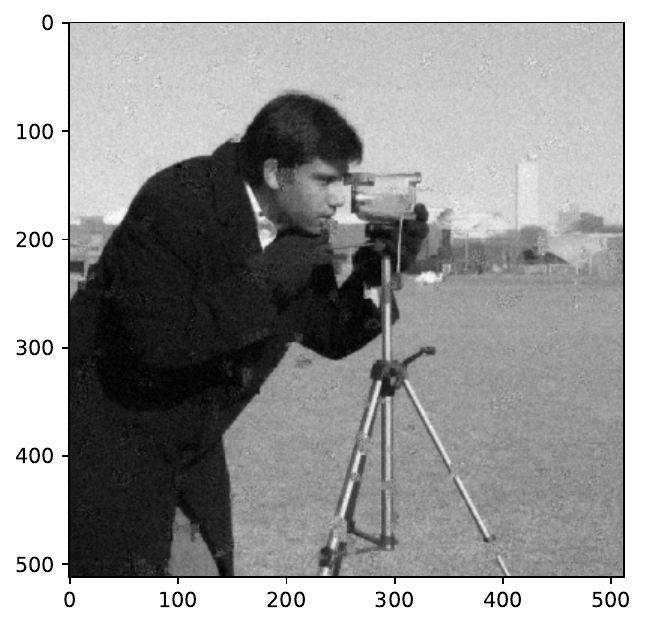}
&\includegraphics[width=0.25\linewidth]{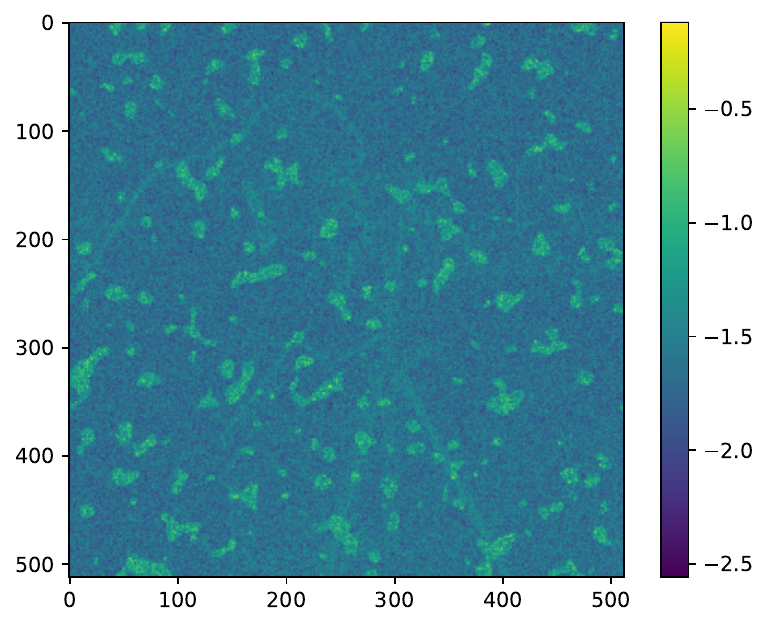}
&
\includegraphics[width=0.22\linewidth]{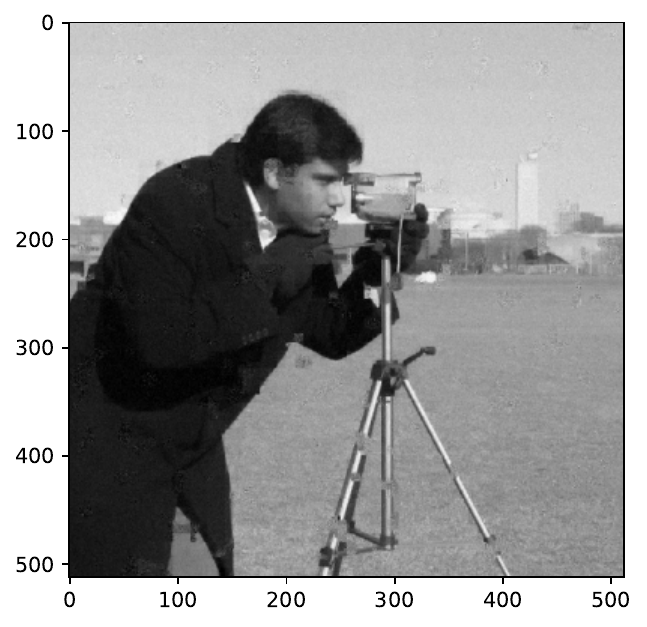}
&
\includegraphics[width=0.25\linewidth]{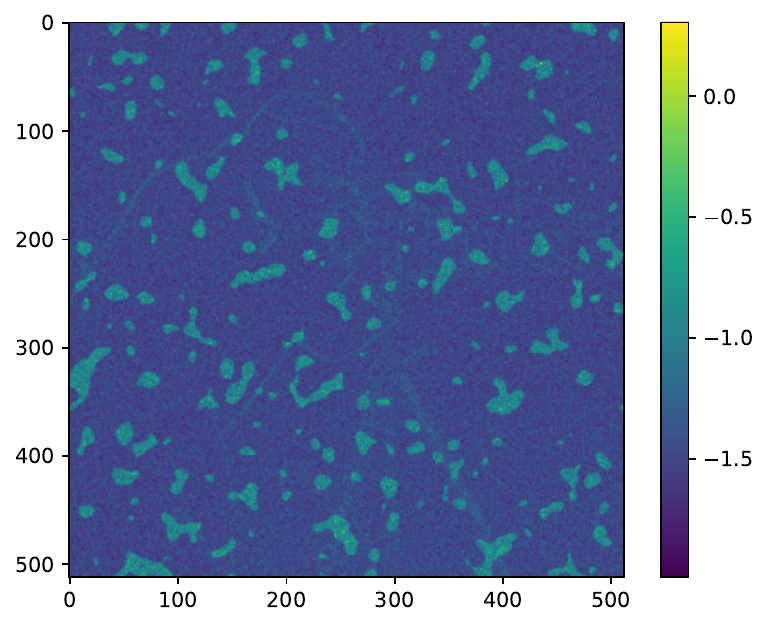}
\\
\rotatebox{90}{\hspace{2em}$\beta=500$}&
\includegraphics[width=0.21\linewidth]{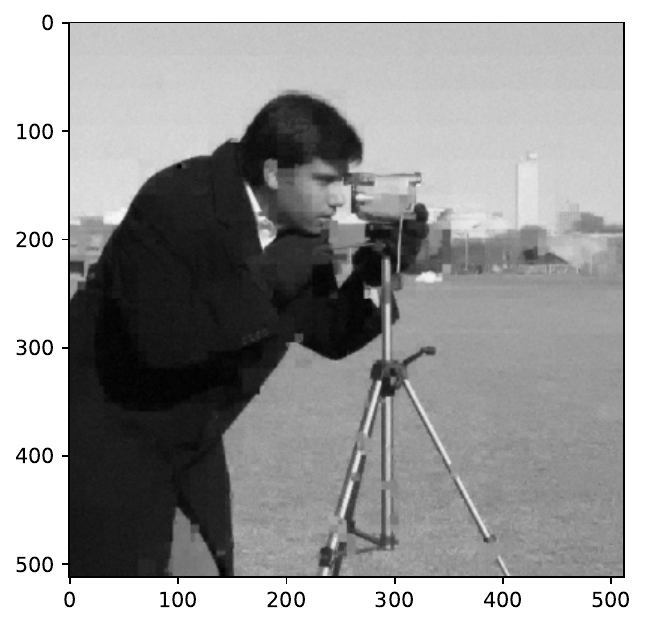}
&\includegraphics[width=0.25\linewidth]{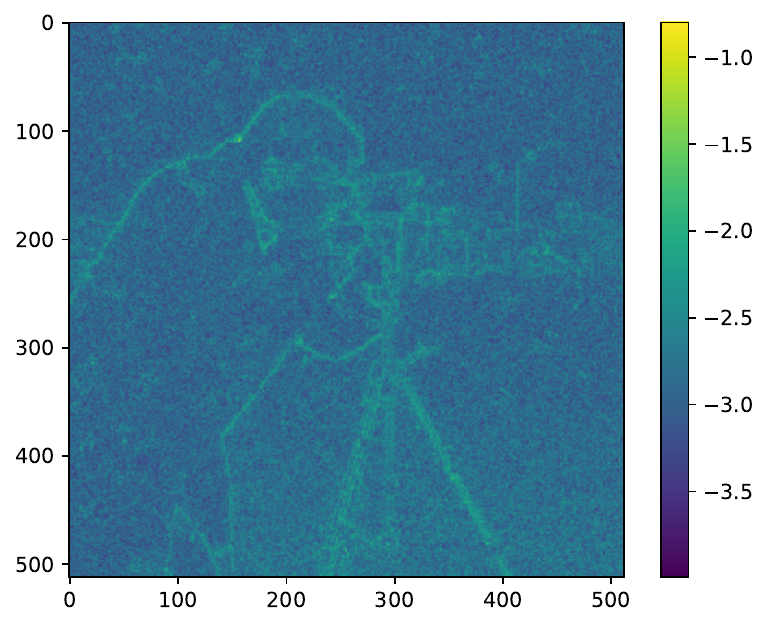}
&
\includegraphics[width=0.21\linewidth]{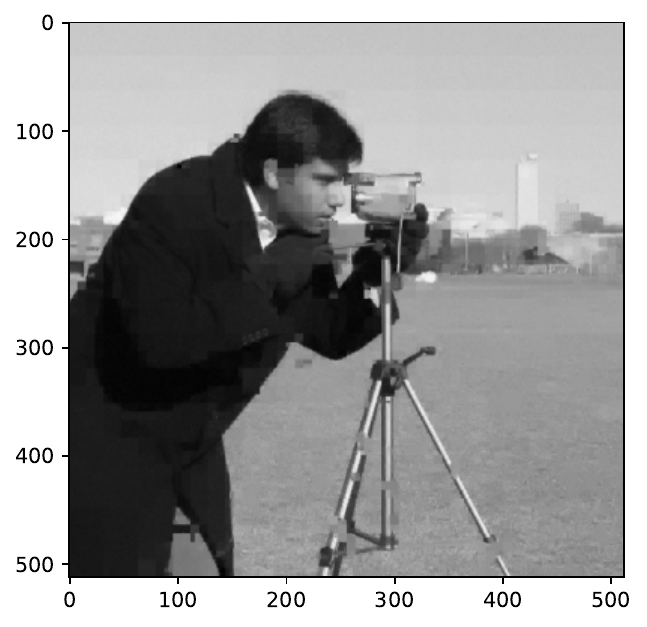}
&
\includegraphics[width=0.25\linewidth]{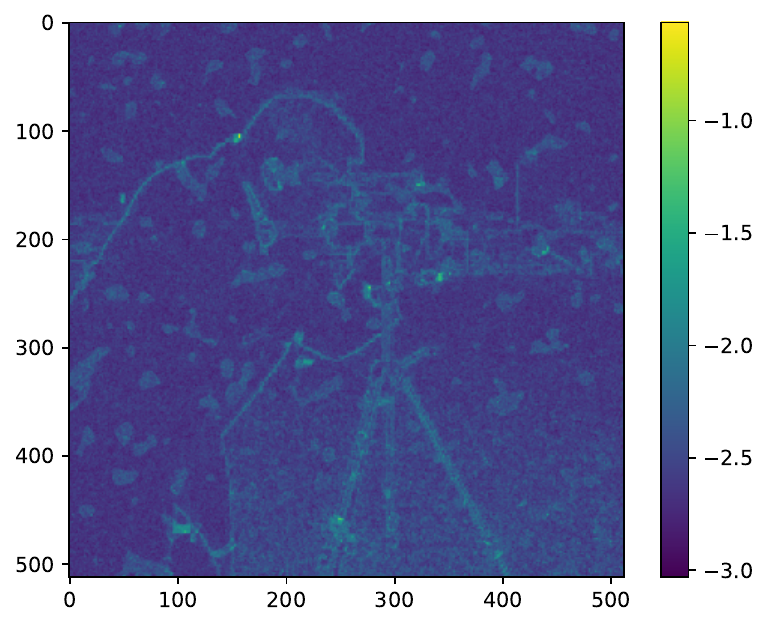}
\end{tabular}    
    \caption{Mean images and  log differences between 95th and 5th quantiles for different choices of $\beta$.}
    \label{fig:inpainting2}
\end{figure}

\end{document}